\documentclass[accepted,hidelinks]{siamart190516}

\usepackage{amssymb}
\usepackage[algo2e,ruled]{algorithm2e}
\usepackage{subcaption}
\usepackage{graphicx}

\newcommand{\bx}{\mathbf{x}}
\newcommand{\by}{\mathbf{y}}
\newcommand{\cO}{\mathcal{O}} 
\newcommand{\cX}{\mathcal{X}} 
\newcommand{\cY}{\mathcal{Y}} 
\newcommand{\cZ}{\mathcal{Z}} 
\newcommand{\cS}{\mathcal{S}} 
\newcommand{\cA}{\mathcal{A}} 
\newcommand{\cB}{\mathcal{B}} 
\newcommand{\cP}{\mathcal{P}} 
\newcommand{\cL}{\mathcal{L}} 
\newcommand{\argmin}{\operatornamewithlimits{argmin}}
\newcommand{\argmax}{\operatornamewithlimits{argmax}}
\newcommand{\RR}{\mathbb{R}}
\newcommand{\R}{\mathbb{R}}

\newcommand{\tby}{\tilde{\by}}

\newtheorem{assumption}[theorem]{Assumption}

\headers{Primal-dual method for affinely constrained minimax problem}{J. Zhang,  M. Wang, M. Hong and S. Zhang}

\title{Primal-Dual First-Order Methods for Affinely Constrained Multi-Block Saddle Point Problems} %\thanks{Submitted to SIAM Journal on Optimization, October 11, 2021.}}
\author{
	Junyu Zhang\thanks{Department of Industrial Systems Engineering and Management, National University of Singapore, junyuz@nus.edu.sg}
	\and
	Mengdi Wang\thanks{Department of Electrical and Computer Engineering, Princeton University, mengdiw@princeton.edu}
	\and
	Mingyi Hong\thanks{Department of Electrical and Computer Engineering, University of Minnesota, mhong@umn.edu}
	\and
	Shuzhong Zhang\thanks{Department of Industrial and Systems Engineering, University of Minnesota, zhangs@umn.edu}
}

\begin{document}
\maketitle 

\begin{abstract} 
We consider the convex-concave saddle point problem $\min_{\bx}\max_{\by}\Phi(\bx,\by)$, where the decision variables $\bx $ and/or $\by$ are subject to certain multi-block structure and affine coupling constraints, and $\Phi(\bx,\by)$ possesses certain separable structure. Although the minimization counterpart of this problem has been widely studied under the topics of ADMM, this minimax problem is rarely investigated. In this paper, a convenient notion of $\epsilon$-saddle point is proposed, under which the convergence rate of several proposed algorithms are analyzed. When only one of $\bx$ and $\by$ has multiple blocks and affine constraint, several natural extensions of ADMM are proposed to solve the problem. Depending on the number of blocks and the level of smoothness, $\cO(1/T)$ or $\cO(1/\sqrt{T})$ convergence rates are derived for our algorithms. When both $\bx$ and $\by$ have multiple blocks and affine constraints, a new algorithm called \underline{E}xtra-\underline{G}radient \underline{M}ethod of \underline{M}ultipliers (EGMM) is proposed. Under desirable smoothness conditions, an $\cO(1/T)$ rate of convergence can be guaranteed regardless of the number of blocks in $\bx$ and $\by$. An in-depth comparison between EGMM (fully primal-dual method) and ADMM (approximate dual method) is made over the multi-block optimization problems to illustrate the advantage of the EGMM.
\end{abstract} 



%%%%%%%%%%%%%%%%%%%%%%%%%%%%%
%Fundings
% Mengdi Wang, NSF Award # 1653435
%%%%%%%%%%%%%%%%%%%%%%%%%%%%%
 
\section{Introduction}
In this paper, we consider the multi-block convex-concave minimax saddle point problems with affine coupling constraints: 
\begin{eqnarray}
	\label{prob:main}
	&\min_{\bx} \max_{\by} &\Phi(\bx,\by) := \sum_{i=1}^N h_i(x_i) + \Psi(\bx,\by) - \sum_{j=1}^M g_j(y_j) \nonumber\\
	&\mathrm{s.t.} & \bx = [x_1^\top,\cdots,x_N^\top]^\top,\quad x_i\in\cX_i\subset\RR^{d_x^i},\quad i = 1,2,\cdots,N,\\
	& & \by = [y_1^\top,\cdots,y_M^\top]^\top,\quad y_j\in\cY_j\subset\RR^{d_y^j},\quad j = 1,2,\cdots,M,\nonumber\\
	& & \!A_1x_1 +  \cdots + A_Nx_N = a,\!\qquad B_1y_1  + \cdots + B_My_M = b.\nonumber
\end{eqnarray}
In problem \eqref{prob:main}, $h_i$ and $g_j$ are simple convex functions that allow efficient proximal operator evaluation. The function $\Psi$ is a smooth convex-concave function that couples the multiple blocks of $\bx$ and $\by$ together. $\cX_i$ and  $\cY_j$ are compact convex sets for $\forall i,j$.  $A_i\in\RR^{n\times d_x^i}, B_j\in\RR^{m\times d_y^j}, \forall i,j$ are a group of matrices and $a\in\RR^{n}$, $b\in\RR^{m}$ are two vectors. The proposed problem lies in the conjunction of the affinely constrained multi-block optimization problem and the convex-concave saddle point problems, which are extensively studied in the alternating direction method of multipliers (ADMM) and the monotone variational inequality (VI) literature respectively. Many works on the saddle point problems do allow convex constraints on the variables. However, they usually assume an easy access to the projection operator to the constraint sets, which is hard to evaluate when there are multiple blocks of variables that are affinely constrained in addition to the convex set constraints. To our best knowledge, our paper is the first one to consider the saddle point problems with affine constraints. 
Though rarely studied in the literature, this problem has many potential applications.
Several motivating examples ranging from multi-agent Reinforcement Learning (RL) to game theory are listed below.\vspace{0.2cm}

\noindent\textbf{Team-collaborative RL.}$~$ Consider the team-collaborative RL setting, where the state space $\cS$ is the nodes in a network and is partitioned into $N$ clusters: $\cS = \cup_{i=1}^N\cS_i$. Each cluster $\cS_i$ has an agent to control its own policy, i.e., $\pi(\cdot|s)$ for $\forall s\in\cS_i$. Let $\cA,\cP$ and $\gamma\in(0,1)$ be the action space, transition probability, and discount factor respectively. Let $\mu^\pi_\xi\in\RR^{|\cS|\times|\cA|}$ be the state-action occupancy measure under some initial state distribution $\xi$ and policy $\pi$, then the goal of the agents is to collaboratively solve a general utility RL problem: $\pi^* = \argmax_{\pi}\sum_{i=1}r_i(\mu^\pi_\xi(\cS_i,:)) + \rho(\mu^\pi_\xi)$, where $r_i(\cdot)$ and $\rho(\cdot)$ are concave functions, see e.g. \cite{zhang2020variational}, and $\mu^\pi_\xi(\cS_i,:)$ stands for the submatrix of $\mu^\pi_\xi$ that includes the row index $s\in\cS_i$. %In details, given an initial state distribution $\xi$, we define for any policy $\pi$ and any state-action pair $(s,a)\!\in\!\cS\!\times\!\cA$ that 
%$$\mu^{\!\pi}_\xi(s,\!a)\!:=\sum_{t=0}^{\infty}\gamma^t\cdot \PP\Big(\!s_t \!=\! s, a_t \!=\! a \mid s_0\!\sim\!\xi, a_t\!\sim\!\pi(\cdot|s_t),s_{t+1}\!\sim\!\cP(\cdot|s_t,a_t)\!\Big).$$ 
To avoid the nonconvexity in terms of $\pi$,  we reformulate it as a convex-concave minimax occupancy optimization problem:\vspace{-0.1cm}
%\begin{eqnarray}
%	\label{prob:rl-occupancy}
%	&\max&\quad \rho(\mu)\nonumber\\
%	&\mathrm{s.t.}&\quad \mu\in\left[0,1/(1-\gamma)\right]^{|\cS|\times|\cA|},\\
%	& & \,\,\,\sum_{a\in\cA}\mu(s,a) = \gamma\!\!\sum_{s'\in\cS,a'\in\cA}\mu(s',a')P(s|s',a') + \xi(s), \,\,\,\,\forall s\in\cS.\nonumber
%\end{eqnarray} 
%In the team-collaborative RL setting, the state space $\cS$ is partitioned into $N$ clusters of nodes in a network $\cS = \cup_{i=1}^N\cS_i$. Each cluster $\cS_i$ has one agent to control its own policy, i.e., $\pi(\cdot|s)$ for $\forall s\in\cS_i$. Let $\mu_{i}:=\{\mu(s,a):s\in\cS_i,a\in\cA\}$ be the occupancy measure on the cluster $\cS_i$. With additional local utility $r_i(\mu_i)$ for each cluster, \eqref{prob:rl-occupancy} is reformulated as 
\begin{eqnarray}
	\label{prob:mdp}
	&&\mathop{\mathrm{maximize}}_{\mu\in\RR_+^{|\cS|\times|\cA|}}\,\,\mathop{\mathrm{minimize}}_{z\in\RR^{|\cS|\times|\cA|}}\,\, \sum_{i=1}^Nr_i\big(\mu(\cS_i,:)\big) + \langle\mu,z\rangle - \rho^*(z)\\
	&&\mathrm{s.t.}\quad \sum_{a\in\cA}\mu(s,a) = \gamma\sum_{s'\in\cS,a'\in\cA}\mu(s',a')P(s|s',a') + \xi(s),\,\,\,\forall s\in\cS,\nonumber
\end{eqnarray}
where the conjugate function $\rho^*$ of $\rho$ is used to decouple the multiple $\mu(\cS_i,:)$ blocks owned by each individual agent respectively.   \vspace{0.3cm}

\noindent\textbf{Resource constrained game.}$~$
The problem \eqref{prob:main} can also be interpreted as several game theory settings, including two-player multi-stage games and multi-player games. \vspace{-0.2cm}

\noindent{\bf (i)} (Two-player multi-stage game). Consider two players playing a sequence of $N-1$ games, with their strategies in the $i$-th stage denoted by $x_i\in\cX_i$ and $y_i\in\cY_i$ respectively. The minimax objective function at stage $i$ takes the form $h_i(x_i) + \Psi_i(\bx,\by) - g_i(y_i)$.  After taking the strategies $x_i$ and $y_i$,  the two players will incur a resource cost of $A_ix_i$ and $B_iy_i$ respectively. The total resource of the two players are limited by two vectors $a$ and $b$ respectively. Therefore, the problem can be written as 
\begin{eqnarray}
	\label{prob:game}
	&&\mathop{\mathrm{minimize}}_{\bx\in\cX_1\times\cdots\times\cX_N}\,\,\mathop{\mathrm{maximize}}_{\by\in\cY_1\times\cdots\times\cY_N}\,\, \sum_{i=1}^{N-1}\Big(h_i(x_i) + \Psi_i(\bx,\by) - g_i(y_i)\Big)\\
	&&\mathrm{s.t.}\quad \sum_{i=1}^{N-1}A_ix_i + x_N = a,\quad \sum_{i=1}^{N-1}B_iy_i + y_N = b,\nonumber
\end{eqnarray}
where $x_N$, $y_N$ are slack variables. Let $B(z,r)$ denote a Euclidean ball centered at $z$ with radius $r$. Then we can specify $\cX_N: = \RR^n_+\cap B(0,r_\bx)$, $\cY_N:=\RR^m_+\cap B(0,r_\by)$.  Simple choices of $r_\bx$ is $r_\bx = \|a\| + \sum_{i=1}^{N-1}\|A_i\|\cdot\max_{x_i\in\cX_i}\|x_i\|$, which  bounds the magnitude of all feasible slack variable $x_N$. The radius $r_\by$ can be constructed similarly.\vspace{0.2cm}

\noindent{\bf (ii)} (Multi-player game). Instead of two players playing an $(N-1)$-stage game, we can consider the cases where two groups of players play a single-stage game. There are $N-1$ and $M-1$ players in the two groups respectively and the players in each group play collaboratively against the players in the other group. The strategy of these players are denoted by $x_i\in\cX_i, 1\leq i\leq N-1$ and $y_j\in\cY_j, 1\leq j\leq M-1$. Each player has its own interest represented by $h_i(x_i)$ or $g_j(y_j)$ as well as a coupled common interest $\Psi(\bx,\by)$. The two groups have their shared resource budget of $a$ and $b$ respectively. The resource cost of taking strategy $x_i$ or $y_j$ are $A_ix_i$ or $B_jy_j$ respectively, for $\forall i,j$. Therefore, the problem can be formulated in a similar form of \eqref{prob:game}, with $x_N$ and $y_M$ being the slack variables, and $\Phi(\bx,\by) = \sum_{i=1}^{N-1}h_i(x_i) + \Psi(\bx,\by) - \sum_{j=1}^{M-1}g_j(y_j).$\vspace{0.4cm}

\noindent\textbf{Related works.}$~$
In this work, the main problem \eqref{prob:main} is mostly related to the affinely constrained multi-block optimization problem and the convex-concave minimax saddle point problem, both of which have long history of research in convex optimization and variational inequalities. The proposed methods are closely related to the extragradient (EG) method and the Alternating Direction Method of Multipliers (ADMM).\vspace{0.2cm}

\noindent\textbf{Multi-block ADMM algorithms.} $~$ As a special case of \eqref{prob:main}, the affinely constrained multi-block convex optimization problem \vspace{-0.1cm}
\begin{equation}
	\label{prob:admm}
	\mathop{\mathrm{minimize}}_{x_i\in\cX_i, 1\leq i\leq N} \,\, \sum_{i=1}^Nh_i(x_i)+\Psi(\bx)\quad\mbox{s.t.}\quad A_1x_1 + A_2x_2 + \cdots + A_Nx_N = a,
\end{equation} 
has vast applications in statistics, signal processing, and distributed computing, see e.g.  \cite{ma2018efficient,lin2013design,chen2018convexified,boyd2011distributed}. Despite the popularity of  ADMM, its convergence is very subtle. When there are only two blocks of variables, i.e. $N = 2$, the convergence of ADMM to optimal solution is well established,  \cite{he20121,monteiro2013iteration,deng2016global,
	lin2015sublinear,hong2017linear}. When $N\geq 3$, counterexamples where  ADMM diverges are constructed \cite{chen2016direct}. Additional assumptions or algorithm modifications are needed to ensure the convergence. For example, with additional error bound \cite{hong2017linear} or the partial strong convexity \cite{chen2013convergence,lin2015sublinear,lin2015global,li2015convergent,cai2017convergence}, convergence to optimal solution can be guaranteed; by making strongly convex $\epsilon$-perturbations to the objective function \cite{lin2016iteration} or certain randomization over the update rule \cite{gao2019randomized}, the convergence of multi-block ADMM can still be achieved without additional conditions. 

\noindent\textbf{Convex-concave  minimax problems.}$~$ Saddle point problems are beyond the scope of the pure optimization. Such problems are crucial in many areas including game theory \cite{von2007theory,nisan2007algorithmic}, reinforcement learning  \cite{dai2018sbeed,chen2018scalable},  and image processing \cite{chambolle2011first},   to name a few. As a special case of the monotone variational inequality problems (VIP), most algorithms derived for monotone VIPs apply directly to the convex-concave saddle point problems. In terms of first-order algorithms, representative classical works include the extragradient (EG) method  \cite{korpelevichextrapolation,censor2011subgradient}, Mirror-Prox method \cite{nemirovski2004prox,juditsky2011solving}, and dual extrapolation methods \cite{nesterov2006solving,nesterov2007dual}, among others. For problems with bilinear coupling objective function,  optimal first-order algorithms have been derived in \cite{chambolle2011first,chambolle2016ergodic,chen2014optimal}, which matched the iteration complexity lower bounds provided by \cite{ouyang2019lower,zhang2019lower}. For general nonlinear coupling problems, near-optimal first-order algorithms \cite{lin2020near,wang2020improved} are also discovered.  More recent algorithmic development in first-order methods include \cite{abernethy2019last,liang2019interaction,mokhtari2020unified}. Recently,  people also start to consider the saddle point problems with multi-block structure. In \cite{xiao2019dscovr}, the authors proposed a stochastic variance reduced block coordinate method for finite-sum saddle point problem with bi-linear coupling. In \cite{jalilzadeh2019doubly}, a randomized block coordinate algorithm for saddle point problem with nonlinear coupling term is proposed. To our best knowledge,  the algorithmic development for affinely constrained multi-block saddle point problems remains unexplored.
\vspace{0.3cm}

\noindent\textbf{{Contributions.}}$~$
We summarize the contributions of the paper as follows.
\begin{itemize}
	\item We propose a concept of the \emph{$\epsilon$-saddle point} for the affinely constrained minimax problem and a handy sufficient condition to guarantee a point to be an $\epsilon$-saddle point.
	\item We consider a simple case of problem \eqref{prob:main} where only $\bx$ has $N=2$ blocks coupled through an affine constraint, while $\by$ is subject to no multi-block structure and affine constraint. Under different smoothness assumptions, we design two algorithms called SSG-ADMM and SEG-ADMM that achieve $\cO(1/\sqrt{T})$ and $\cO(1/T)$ convergence rates respectively. The analysis framework for SSG-ADMM and SEG-ADMM is very versatile. We show that it can easily incorporate classical multi-block ($N\geq3$) ADMM analysis, given additional conditions or proper algorithm modification.
	\item We consider the general case of \eqref{prob:main}, where both $\bx$ and $\by$ have multiple blocks coupled through affine constraints. An EGMM algorithm is proposed to solve the general problem \eqref{prob:main} with an $\cO(1/T)$ convergence rate. Unlike the ADMM-type algorithms SSG-ADMM and SEG-ADMM, EGMM is fully primal-dual and it abandons the augmented Lagrangian terms. EGMM not only keeps the benefit of SEG-ADMM and SEG-ADMM in solving small separable subproblems, but also guarantees the convergence regardless of the number of blocks.
	\item Under the special case of multi-block minimization problem, we make an extensive comparison between EGMM (primal-dual method) and ADMM (approximate dual ascent) to illustrate the benefits of the primal-dual methods. \vspace{0.0cm}
\end{itemize}

\noindent\textbf{Organization.}$~$ 
In Section \ref{section:epsolu}, we introduce the definition of an $\epsilon$-saddle point as well as a convenient sufficient condition. In Section \ref{section:1side-2block}, we propose the SSG-ADMM and SEG-ADMM algorithms and derive their convergence rate for solving a special case of problem \eqref{prob:main}. In Section \ref{section:two-sided}, we propose and analyze the EGMM algorithm in solving the general case problem \eqref{prob:main}. In Section \ref{section:discussion}, we make extensive comparison between EGMM and ADMM to illustrate the advantage of the primal-dual methods. Part of the discussion and proof has been moved to the Appendix. \vspace{0.3cm}

\noindent\textbf{Notations.}$~$
For the ease of notation, we will often write $\mathbf{x} = [x_1^\top,\cdots,x_N^\top]^\top$, and $\mathbf{x}_{i:j} = [x_{i}^\top,\cdots,x_{j}^\top]^\top$ for $i\leq j$. Similarly, we write $\mathbf{y} = [y_1^\top,\cdots,y_M^\top]^\top$, and $\mathbf{y}_{i:j} = [y_{i}^\top,\cdots,y_{j}^\top]^\top$ for $i\leq j$. We denote the dimension of $a$ as $n$ and the dimension of $b$ as $m$, i.e. $a\in\RR^n$ and $b\in\RR^m$. We also define $\cX = \cX_1\times\cdots\times\cX_N$ and $\cA(\mathbf{x}) := A_1x_1 + A_2x_2 + \cdots + A_Nx_N$. Similarly, we also define $\cY = \cY_1\times\cdots\times\cY_M$ and $\cB(\by) := B_1y_1 + B_2y_2 + \cdots + B_My_M$. In some situations, it will be more convenient to write $A = [A_1,\cdots,A_N]$, $B = [B_1,\cdots,B_N]$ and $\cA(\mathbf{x}) = A\bx$, $\cB(\by) = B\by$. We often switch between the two notations. We also write $h(\bx) = \sum_{i=1}^N h_i(x_i)$ and $g(\by) = \sum_{j = 1}^Mg_j(y_j)$. Therefore, we can also write  \eqref{prob:main} in a more compact form:
\begin{eqnarray}
	\label{prob:main-simple}
	&\min_{\mathbf{x}\in\cX}\max_{\by\in\cY}& \Phi(\bx,\by):= h(\bx) + \Psi(\bx,\by) - g(\by)\\
	&\mathrm{s.t.}& A\bx= a, B\by = b.\nonumber
\end{eqnarray} 
For the compact convex sets $\cX_i, i = 1,\cdots,N$ and $\cY_j, j = 1,\cdots,M$, we denote their diameters as $D_{\cX_i} = \max\{\|x_i-x'_i\|: x_i,x_i'\in\cX_i\}<+\infty$, $D_{\cY_j} = \max\{\|y_j-y_j'\|: y_j,y_j'\in\cY_j\}<+\infty$. We also denote $D_\cX = \sqrt{\sum_{i=1}^ND_{\cX_i}^2}$ and $D_\cY = \sqrt{\sum_{j=1}^MD_{\cY_j}^2}$.

\section{The $\epsilon$-saddle point condition} 
\label{section:epsolu}
As the first step of solving problem \eqref{prob:main}, let us study the definition of an $\epsilon$-saddle point. For convex-concave minimax saddle point problems, the classical concept of an $\epsilon$-saddle point is often defined as a feasible solution where the duality gap is bounded by $\epsilon$. That is, a point $(\bar\bx,\bar\by)$ such that $A\bar\bx = a,B\bar\by = b,\bar\bx\in\cX,  \bar\by\in\cY$ and 
\begin{equation*}
	\Delta(\bar{\bx},\bar{\by}):=\max_{\overset{\mathbf{y}\in\cY}{B\by = b}} \Phi(\bar{\mathbf{x}}, \by) \!-\!\! \min_{\overset{\mathbf{x}\in\cX}{A\bx= a}}\Phi(\mathbf{x},\bar \by) \leq \epsilon.
\end{equation*} 
By weak-duality theorem, $\Delta(\bar{\bx},\bar{\by})\geq0$ always holds as long as $(\bar\bx,\bar\by)$ is feasible. Such definition of $\epsilon$-saddle point is more suitable for measuring the convergence for algorithms that keep the iterates feasible throughout the iterations.  However, for problem \eqref{prob:main}, keeping the feasibility of the affine coupling constraints throughout the iterations can be prohibitively expensive.  $\Delta(\bar{\bx},\bar{\by})$ may even be negative when $\bar{\bx}$ and $\bar{\by}$ are infeasible, which makes the common definition of an $\epsilon$-saddle point meaningless. Therefore, we adopt the following definition of an $\epsilon$-saddle point that allows an $\epsilon$ constraint violation. 
\begin{definition}[$\epsilon$-saddle point]
	\label{definition:eps}
	We say $(\bar{\mathbf{x}}, \bar \by)\in\cX\times\cY$ is an $\epsilon$-saddle point of the multi-block affinely constrained minimax  problem \eqref{prob:main}, if
	\begin{eqnarray*}  
		\max_{\overset{\mathbf{y}\in\cY}{B\mathbf{y} = b}} \Phi(\bar{\mathbf{x}}, \by) \!-\!\! \min_{\overset{\mathbf{x}\in\cX}{A\mathbf{x} = a}}\Phi(\mathbf{x},\bar \by) \in [-\epsilon,\epsilon]\qquad\mbox{and}\qquad
		\|A\bar{\mathbf{x}} - a\| \leq \epsilon,\,\, \|B\bar{\mathbf{y}} - b\| \leq \epsilon.
	\end{eqnarray*}
\end{definition}
%When $(\bar{\bx},\bar{\by})$ is infeasible to the affine constraints, the ``duality gap'' $\Delta(\bar{\bx}, \bar{\by})$ may not be nonnegative, which is the reason why we require $\Delta(\bar{\bx}, \bar{\by})\in[-\epsilon,+\epsilon]$ in Definition \ref{definition:eps}. 
It is worth noting that Definition \ref{definition:eps} reduces to the commonly used $\epsilon$-optimal solution in the convex ADMM literature when problem \eqref{prob:main} takes the special form of \eqref{prob:admm}. For any $\bx\in\cX$ and $\by\in\cY$, let us define the functions $p_\by(\cdot)$ and $q_\bx(\cdot)$ as
\begin{eqnarray}
	\label{defn:qp}
	p_\by(v) := \min_{\overset{\mathbf{x}\in\cX}{A{\mathbf{x}} - a = v}}\Phi(\mathbf{x},\by) \qquad\mbox{and}\qquad q_\bx(u) := \max_{\overset{\mathbf{y}\in\cY}{B\mathbf{y} - b = u}} \Phi({\mathbf{x}}, \by).
\end{eqnarray} 
Then $p_\by(v)$ is a convex function in $v$, and $q_\bx(u)$ is a concave function in $u$, see \cite{bertsekas1997nonlinear}. To facilitate the convergence analysis, we require Slater's condition to hold.
\begin{assumption}[Slater's condition]
	\label{assumption:slater}
	There exists $(\hat \bx,\hat \by)\in\mathrm{int}(\cX)\times \mathrm{int}(\cY)$ s.t. $A\hat \bx = a$, $B\hat \by = b$, where $\mathrm{int}(\cdot)$ denotes the interior of a set. 
\end{assumption} 
As a result of Slater's condition, we have the following lemma. 
\begin{lemma}
	\label{assumption:bounded_rho}
	Let $\partial p_\by(\cdot)$ denote the subgradient of the convex function $p_\by(\cdot)$ and let $\partial q_\bx(\cdot)$ denote the supergradient of the concave function $q_\bx(\cdot)$. Suppose $\Phi(\cdot,\by)$ is convex and lower semi-continuous in $\cX$ for $\forall\by\in\cY$, $\Phi(\bx,\cdot)$ is concave and upper semi-continuous in $\cY$ for $\forall\bx\in\cX$, and $\Phi(\cdot,\cdot)$ is bounded over $\cX\times\cY$. If Assumption 2.2 holds true, then there exists a positive constant $\rho^*>0$ s.t.
	$$\sup_{\by\in\cY}\,\inf_{\omega_1\in\partial p_\by(0)}\|\omega_1\|\,\,\leq\,\, \rho^*\qquad\mbox{and}\qquad\,\, \sup_{\bx\in\cX}\,\inf_{\omega_2\in\partial q_\bx(0)}\|\omega_2\| \,\,\leq\,\, \rho^*.$$
\end{lemma}

The proof of this lemma is placed in Appendix \ref{appdx:Slater}. Next, we introduce a handy sufficient condition for claiming a point $(\bar{\bx},\bar{\by})$ to be an $\epsilon$-saddle point.
\begin{lemma}
	\label{lemma:opt-cond} 
	Let $\rho^*$ be defined by Lemma \ref{assumption:bounded_rho}. Then any $(\mathbf{\bar x}, \bar \by)\in\cX\times\cY$ satisfying
	\begin{eqnarray}
		\label{lm:opt-cond-1}
		\max_{\overset{\mathbf{y}\in\cY}{B\mathbf{y} = b}} \Phi(\bar{\mathbf{x}}, \by) \!-\!\! \min_{\overset{\mathbf{x}\in\cX}{A\mathbf{x} = a}}\Phi(\mathbf{x},\bar \by) +  \rho\|A\bar{\mathbf{x}} - a\| + \rho\|B\bar{\mathbf{y}} - b\|\leq \epsilon
	\end{eqnarray} 
	for some $\rho>\rho^*$ will be an $\cO(\epsilon)$-saddle point of \eqref{prob:main}.
\end{lemma}

\begin{proof}
	First, by adding and subtracting the term $\Phi(\bar{\mathbf{x}}, \bar \by)$, 
	\begin{eqnarray*}
		\max_{\overset{\mathbf{y}\in\cY}{B\mathbf{y} = b}} \Phi(\bar{\mathbf{x}}, \by) \!\!&-&\!\!\! \min_{\overset{{\mathbf{x}}\in\cX}{A{\mathbf{x}} = a}}\Phi(\mathbf{x},\bar \by)\\
		& = & \bigg(\max_{\overset{\mathbf{y}\in\cY}{B\mathbf{y} = b}} \Phi(\bar{\mathbf{x}}, \by)  - \Phi(\bar{\mathbf{x}}, \bar \by)\bigg) + \bigg( \Phi(\bar{\mathbf{x}}, \bar \by) - \min_{\overset{{\mathbf{x}}\in\cX}{A{\mathbf{x}} = a}}\Phi(\mathbf{x},\bar \by)\bigg).
	\end{eqnarray*}
	Consider the term $\Phi(\bar{\mathbf{x}}, \bar \by) - \min_{\mathbf{x}\in\cX,A{\mathbf{x}} = a}\Phi(\mathbf{x},\bar \by)$. Because $A\bar{\mathbf{x}} = a$ does not necessarily hold, the point $\bar{\mathbf{x}}$ may not be feasible. Therefore, we cannot simply argue that this term is non-negative. Let $p_{\bar{\by}}(\cdot)$ be defined by \eqref{defn:qp} with $\by=\bar{\by}$.  Then, by Lemma \ref{assumption:bounded_rho}, there exists $\omega_1\in\partial p_{\bar{\by}}(0)$ s.t. $\|\omega_1\|\leq\rho^*$.
	%Define the function
	%$$p(v) = \min_{\overset{\mathbf{x}\in\cX}{A{\mathbf{x}} - a = v}}\Phi(\mathbf{x},\bar \by).$$
	%Then the function $p(v)$ is a convex function in $v$, see \cite{bertsekas1997nonlinear}. Let $\omega_1\in\partial p(0)$ be a subgradient, 
	Take $\bar v = A\bar{\mathbf{x}} - a$, we have 
	\begin{eqnarray}
		\label{lm:opt-cond-3}
		\Phi(\bar{\mathbf{x}},\bar \by) - \min_{\overset{\mathbf{x}\in\cX}{A\mathbf{x} = a}}\Phi(\mathbf{x},\bar \by) &\geq& \min_{\overset{\mathbf{x}\in\cX}{A\mathbf{x} - a = \bar v}}\Phi(\mathbf{x}, \bar \by) - \min_{\overset{\mathbf{x}\in\cX}{A\mathbf{x} - a = 0}}\Phi(\mathbf{x},\bar \by)\nonumber\\
		& = & p_{\bar{\by}}(\bar v) - p_{\bar{\by}}(0)\\
		& \geq & \langle\omega_1,\bar v-0\rangle\nonumber\\
		& \geq & -\|\omega_1\|\cdot\|A\bar{\mathbf{x}} - a\|.\nonumber
	\end{eqnarray}
	Similarly, there exists $\omega_2\in\partial q_{\bar{\bx}}(0)$ satisfying $\|\omega_2\|\leq\rho^*$ such that  
	\begin{eqnarray}
		\label{lm:opt-cond-2}
		\max_{\overset{\mathbf{y}\in\cY}{B\mathbf{y} = b}} \Phi(\bar{\mathbf{x}}, \by)  - \Phi(\bar{\mathbf{x}}, \bar \by) 
		& \geq & -\|\omega_2\|\cdot\|B\bar{\mathbf{y}} - b\|.
	\end{eqnarray} 
	Combining \eqref{lm:opt-cond-1}, \eqref{lm:opt-cond-3} and \eqref{lm:opt-cond-2}, we get 
	\begin{eqnarray}
		\epsilon & \geq & (\rho-\|\omega_1\|)\|A\bar{\mathbf{x}} - a\| +  (\rho-\|\omega_2\|)\|B\bar{\mathbf{y}} - b\|\nonumber\\
		& & +  \underbrace{\bigg(\max_{\overset{\mathbf{y}\in\cY}{B\mathbf{y} = b}} \Phi(\bar{\mathbf{x}}, \by) \!-\! \min_{\overset{\mathbf{x}\in\cX}{A\mathbf{x} = b}}\!\Phi(\mathbf{x},\bar z)  + \|\omega_1\|\!\cdot\!\|A\bar{\mathbf{x}} - a\| + \|\omega_2\|\!\cdot\!\|B\bar{\mathbf{y}} - b\|\bigg)}_{\geq0}\nonumber.
	\end{eqnarray} 
	Because $\rho>\rho^*$ and $\max\left\{\|\omega_1\|,\|\omega_2\|\right\}\leq \rho^*$, we have $\rho-\|\omega_1\|>0$ and $\rho-\|\omega_2\|>0$, which further implies that $(\bar{\mathbf{x}},\bar \by)$ is an $\cO(\epsilon)$-saddle point described by Definition \ref{definition:eps}.
\end{proof}

\section{One-sided affinely constrained problems}
\label{section:1side-2block}
$~$Before solving the general form problem \eqref{prob:main}, we  consider a slightly simpler setting:
\begin{equation} 
	\label{prob:1-sided}
	\min_{\bx\in\cX} \max_{\by\in\cY} \,\,\,\Phi(\bx, \by) := \sum_{i=1}^Nh_i(x_i,\by) + \Psi(\bx,\by)\qquad\mbox{s.t.}\qquad A\bx = a,	
\end{equation} 
where only $\bx$ is subject to the multi-block structure and the affine coupling constraint. In contrast to the main problem \eqref{prob:main}, we also make \eqref{prob:1-sided} a bit more general by allowing $h_i(x_i,\by)$ to depend on $\by$. 

For problems with this structure, we develop the SSG-ADMM and SEG-ADMM algorithms which naturally extend the well-studied ADMM algorithm to the minimax setting. We will show that the analysis of the proposed methods is very versatile and can easily incorporate the existing results of ADMM research. On the other hand, these methods also suffer from the fundamental restrictions of all ADMM-type algorithms. That is, the methods in general diverge when $N\geq3$. Moreover, such ADMM-type algorithm is extremely hard to analyze for the general problem \eqref{prob:main} which will be solved with a new approach later. 

In this section, we will only discuss the SSG-ADMM and SEG-ADMM algorithms for solving problem \eqref{prob:1-sided} with $N=2$, while leaving the discussion for $N\geq3$ to the Appendix. 
The following assumptions are made for the objective function
under different scenarios. 

\begin{assumption}
	\label{assumption:Convexity-1sided}
	For $\forall\mathbf{x}\!\in\!\cX$, $\Phi(\mathbf{x},\cdot)$ is concave and upper semi-continuous in $\cY$. For $\forall \by\!\in\!\cY$, $h_i(\cdot,\by)$ is convex and lower semi-continuous in $\cX_i$, for $i = 1,...,N$. The overall function $\Phi(\cdot,\cdot)$ takes bounded value over $\cX\times\cY$. For $\forall\by\!\in\!\cY$, $\Psi(\cdot,\by)$ is smooth and convex in $\cX$, and $\nabla_\bx \Psi(\cdot,\by)$ is $L_\bx$-Lipschitz continuous:
	$$\|\nabla_\bx \Psi(\bx,\by) - \nabla_\bx\Psi(\bx',\by)\|\leq L_\bx\|\bx-\bx'\|,\quad\forall\bx,\bx'\in\cX, \forall\by\in\cY.$$
\end{assumption}
\noindent If $\Phi(\mathbf{x},\cdot)$ is nonsmooth, the following assumption is made. 
\begin{assumption}
	\label{assumption:bounded-supgrad-1sided}
	The supergradient of $\Phi(\mathbf{x},\cdot)$ w.r.t. $\by$ is upper bounded by some constant $0\leq\ell<+\infty$. Namely,  $\sup\big\{\|u\|: u\in\partial_\by \Phi(\mathbf{x},\by), \mathbf{x}\in\cX, \by\in\cY\big\} \leq  \ell$.
\end{assumption}

If $\Phi(\mathbf{x},\cdot)$ is smooth, we make the following assumption. 

\begin{assumption}
	\label{assumption:Lipschitz-1sided}
	The partial gradient  $\nabla_{\!\by} \Phi(\cdot,\cdot)$ is $L_\by$-Lipschitz continuous: 
	$$\|\nabla_{\!\by} \Phi(\bx,\by) - \nabla_{\!\by} \Phi(\bx',\by')\|\leq L_\by\sqrt{\|\bx-\bx'\|^2+\|\by-\by'\|^2},\quad \forall \bx,\bx'\in\cX, \forall \by,\by'\in\cY.$$
\end{assumption}

Define the \emph{linearized} augmented Lagrangian function as 
\begin{equation}
	\label{defn:AugLag}
	\cL_\gamma(\mathbf{x},\lambda;\tilde{\bx},\tilde\by) = \sum_{i=1}^Nh_i(x_i,\tilde\by) + \langle\nabla_\bx\Psi(\tilde\bx,\tilde{\by}),\bx-\tilde{\bx}\rangle - \langle\lambda,A\mathbf{x}-a\rangle + \frac{\gamma}{2}\|A\mathbf{x}-a\|^2.
\end{equation}
Based on this notation, we introduce the \texttt{Prox-ADMM} module described by Algorithm \ref{alg:ADMM-module}, which is a common ingredient of both SSG-ADMM and SEG-ADMM.\\  
\begin{algorithm2e}[H]
	\DontPrintSemicolon  
	\caption{Proximal ADMM step $(\mathbf{x}^+\!, \lambda^+) \!\!=\!\! \texttt{Prox-ADMM}(\mathbf{x},\!\lambda;\by;\!\gamma,\!\{\!H_i\!\}_{i\!=\!1}^{\!N})$}
	\label{alg:ADMM-module}
	\vspace{0.1cm}\textbf{input}: ${\mathbf{x}}\in\cX$, $ \by\in\cY$, $\lambda\in\R^n$, $\gamma>0$, and matrices $H_i\succ 0$, $i=1,2,...,N.$\\
	\vspace{0.1cm}Update the decision variable $\mathbf{x}^+$:\vspace{-0.1cm}
	$$\begin{cases}
		x_1^+ = \argmin_{w_1\in\cX_1} \cL_{\gamma}\left(w_1,{\mathbf{x}}_{2:N}, \lambda; \bx,\by\right) + \frac{1}{2}\|w_1-x_1\|^2_{H_1}\\
		x_2^+ = \argmin_{w_1\in\cX_2} \cL_{\gamma}\left(x_1^+, w_2,{\mathbf{x}}_{3:N}, \lambda; \bx,\by\right) + \frac{1}{2}\|w_2 - x_2\|^2_{H_2}\\
		\qquad\qquad\qquad\qquad\qquad\quad\vdots\\
		x_i^+ = \argmin_{w_i\in\cX_i} \cL_{\gamma}\left(\mathbf{x}_{1:i-1}^+, w_i,{\mathbf{x}}_{i+1:N}, \lambda; \bx,\by\right) + \frac{1}{2}\|w_i - x_i\|^2_{H_i}\\
		\qquad\qquad\qquad\qquad\qquad\quad\vdots\\
		x_N^+ = \argmin_{w_N\in\cX_N} \cL_{\gamma}\left(\mathbf{x}_{1:N-1}^+, w_N, \lambda; \bx,\by\right) + \frac{1}{2}\|w_N - x_N\|^2_{H_N}
	\end{cases}$$\\
	Update the Lagrangian multipliers: $\lambda^+ = \lambda-\gamma\cdot (A\mathbf{x}^+-a)$\\
	\textbf{output}: $(\mathbf{x}^+,\lambda^+)$.
\end{algorithm2e} 
\noindent We assume the subproblems in Algorithm \ref{alg:ADMM-module} can be solved efficiently. In particular, if we set the positive definite matrices as $H_i = \sigma I  - \gamma A_i^\top A_i$, the subproblem becomes 
$$\min_{w_i\in\cX_i} h_i(w_i,\tilde\by) + \frac{\sigma}{2}\bigg\|w_i - \Big[x_i - \frac{1}{\sigma}\Big(\nabla_{\!\bx_i}\Psi(\bx,\by)-A_i^\top\lambda + \gamma\Big(\sum_{j<i}A_jx_j^+ + \sum_{j\geq i}A_jx_j-a\Big)\Big)\Big]\bigg\|^2,$$	
which can be viewed as linearizing the augmented quadratic penalty term. Next, we characterize the iteration of the \texttt{Prox-ADMM} module by Lemma \ref{lemma:2-ADMM-step}.
\begin{lemma}
	\label{lemma:2-ADMM-step}
	For problem \eqref{prob:1-sided}, suppose Assumption \ref{assumption:Convexity-1sided} holds and $N = 2$. Let $(\mathbf{x}^{k+1},\lambda^{k+1}) = \emph{\texttt{Prox-ADMM}}(\mathbf{x}^k\!,\lambda^k;\tilde \by;\gamma,\{H_i\}_{i=1}^N)$ for some $\tilde \by\!\in\!\cY\!$. Define the  block diagonal matrix $ H \!:=\! \mathrm{Diag}(H_1,...,H_N)$. 	Then, for $\forall\lambda$ and $\forall\mathbf{x}\!\in\!\cX\!$ s.t. $A\mathbf{x} \!=\! a$, we have
	\begin{align} 
		\label{lm:2-ADMM-step}
		\Phi(\bx^{k+1},\tilde \by) -& \Phi(\bx,\tilde \by) -   \langle\lambda,A\bx^{k+1}-a\rangle \leq  \frac{1}{2\gamma}\left(\big\|\lambda-\lambda^k\big\|^2-\big\|\lambda-\lambda^{k+1}\big\|^2\right)\nonumber\\
		\leq \,\,&
		\frac{\gamma}{2}\left(\big\|A_1x_1+A_2x_2^k-a\big\|^2-\big\|A_1x_1+A_2x_2^{k+1}-a\big\|^2\right)  \nonumber \\
		& + \frac{1}{2}\left(\|\bx-\bx^k\|^2_H - \|\bx-\bx^{k+1}\|^2_H\right) - \frac{1}{2}\|\bx^k-\bx^{k+1}\|^2_{H-L_\bx I}.
	\end{align}
\end{lemma}
The analysis of this lemma mainly combines the techniques from \cite{lin2015sublinear} and \cite{gao2019randomized}. The proof is presented in Appendix \ref{appdx:lemma:2-admm-step}.

\subsection{The  SSG-ADMM for nonsmooth problem}
In this section, we consider the setting where Assumptions \ref{assumption:Convexity-1sided} and \ref{assumption:bounded-supgrad-1sided} hold. Namely, $\Phi(\bx,\cdot)$ can be a nonsmooth function of $\by$. Under these conditions, we propose the following \underbar{S}addle-point \underbar{S}uper\underbar{G}radient \underbar{A}lternating \underbar{D}irection \underbar{M}ethod of \underbar{M}ultipliers (SSG-ADMM), see Algorithm \ref{alg:SSG-ADMM}, which is able to obtain an $\cO(1/\sqrt{T})$ rate of convergence.

\begin{algorithm2e}
	\DontPrintSemicolon  
	\caption{The  SSG-ADMM Algorithm}
	\label{alg:SSG-ADMM}
	\textbf{input}: $\mathbf{x}^0\in\cX, \by^0\in\cY$, $\lambda^0=0\in\RR^n$, $\gamma>0$. Matrices $G\succ0$ and $H_i\succ 0$.  \\
	\For{$k = 0,...,T$}{
		Apply the \texttt{Prox-ADMM} submodule:
		$$(\mathbf{x}^{k+1},\lambda^{k+1}) = \texttt{Prox-ADMM}(\mathbf{x}^k,\lambda^k;\by^k;\gamma, \{H_i\}_{i=1}^N).$$\\
		Compute $u^k\in\partial_\by \Phi(\bx^{k+1},\by^{k})$ and apply supergradient ascent step:\vspace{-0.1cm} 
		$$\by^{k+1} = \argmin_{\by\in\cY} \,\,\frac{1}{2}\left\|\by-[\by^k + G^{-1} u^k]\right\|^2_G.$$\vspace{-0.2cm}}
	\textbf{output:} $\bar \bx = \frac{1}{T}\sum_{k = 1}^T\mathbf{x}^k$, $\bar \by = \frac{1}{T}\sum_{k = 0}^{T-1}\by^k$.
\end{algorithm2e} 
In the  SSG-ADMM algorithm, we apply one \texttt{Prox-ADMM} step to update $\bx$ and one supergradient  step to update $\by$. While $\bx$ update is directly characterized by Lemma \ref{lemma:2-ADMM-step}, we analyze $\by$ update in the following lemma, see proof in Appendix \ref{appdx:lemma:SG-step}.

\begin{lemma}
	\label{lemma:SG-step}
	Let $\by^k$ and $\by^{k+1}$ be generated by Algorithm \ref{alg:SSG-ADMM}. For $\forall\by\in\cY$, we have 
	\begin{eqnarray}
		\label{lm:SG-step}
		\quad\quad\Phi(\bx^{k+1},\by) - \Phi(\bx^{k+1},\by^k) \leq \frac{1}{2}\Big(\|\by-\by^{k}\|^2_G - \|\by-\by^{k+1}\|^2_G\Big) + \frac{1}{2}\|u^k\|^2_{G^{-1}}.\!\!
	\end{eqnarray} 
\end{lemma}
As a result, we have the following convergence rate result. 
\begin{theorem}[Convergence of SSG-ADMM]
	\label{theorem:complexity-ssg}
	Consider problem \eqref{prob:1-sided} with $N=2$. Suppose Assumptions \ref{assumption:Convexity-1sided} and \ref{assumption:bounded-supgrad-1sided} hold and $(\bar{\bx},\bar \by)$ is returned by Algorithm \ref{alg:SSG-ADMM} after $T$ iterations.  
	If we choose $H_i\succeq L_\bx\cdot I, \forall i$ and $G\succeq \frac{\sqrt{T}\ell}{D_\cY}\cdot I$, it holds for $\forall\rho>0$ that
	\begin{eqnarray}
		\label{thm:complexity-SSG}
		\max_{\by\in \cY} \Phi(\bar{\bx},\by) - \!\!\!\! \min_{\bx\in\cX,A\bx= a}\!\!\Phi(\bx,\bar \by)  +  \rho\|A\bar{\bx}\!-\! a\|\leq \frac{\rho^2/\gamma+ \|H\| D_{\cX}^2 + \gamma\|A_2\|^2 D_{\cX_2}^2}{2T} + \frac{\ell D_\cY}{\sqrt{T}}.\nonumber
	\end{eqnarray}
	In particular, if we set $H_i = L_\bx\cdot I$ for $\forall i$ and $G = \frac{\sqrt{T}\ell}{D_\cY}\cdot I$, by Lemma \ref{lemma:opt-cond}, it takes $T=\cO(\epsilon^{-2})$ iterations to reach an $\epsilon$-saddle point.
\end{theorem}
\begin{proof}
	First, let us sum up \eqref{lm:2-ADMM-step} (with $\tilde{\by} = \by^k$ and $\{H_i\}_{i=1}^N$) and \eqref{lm:SG-step}:
	\begin{eqnarray}
		& & \Phi(\bx^{k+1},\by) - \Phi(\bx,\by^{k}) - \langle\lambda, A\bx^{k+1} - a\rangle\nonumber\\
		&\leq& \frac{1}{2\gamma}\!\left(\big\|\lambda\!-\!\lambda^k\big\|^2\!-\!\big\|\lambda-\lambda^{k+1}\big\|^2\right) \!+\! \frac{\gamma}{2}\!\left(\big\|A_1x_1\!+\!A_2x_2^k\!-\!a\big\|^2\!-\!\big\|A_1x_1\!+\!A_2x_2^{k+1}\!-\!a\big\|^2\right)\nonumber\\
		& & + \frac{1}{2}\left(\|\bx-\bx^k\|^2_H - \|\bx-\bx^{k+1}\|^2_H\right)\!+\! \frac{1}{2}\!\left(\|\by-\by^k\|^2_G - \|\by - \by^{k+1}\|^2_G\right) \!+\! \frac{1}{2}\big\|u^k\big\|^2_{G^{-1}}.\nonumber
	\end{eqnarray}
	Averaging the above inequality for $k = 0,...,T-1$, then Jensen's inequality implies
	\begin{eqnarray}
		& & \Phi(\bar{\mathbf{x}},\by) - \Phi(\bx,\bar \by) - \langle\lambda, A\bar{\bx} - a\rangle\nonumber\\
		&\leq&\frac{1}{T}\sum_{k = 0}^{T-1}\Big(\Phi(\bx^{k+1},\by) - \Phi(\bx,\by^{k}) - \langle\lambda, A\bx^{k+1} - a\rangle\Big)\nonumber\\
		&\leq& \frac{\|\lambda-\lambda^0\|^2}{2\gamma\cdot T} + \frac{\gamma\|A_1x_1+A_2x_2^0-a\|^2+\|\bx-\bx^0\|^2_H+\|\by-\by^0\|^2_G}{2T} + \frac{\sum_{k = 0}^{T-1}\|u^k\|^2_{G^{-1}}}{2 T}\nonumber\\
		&  \leq & \frac{\|\lambda-\lambda^0\|^2}{2\gamma\cdot T} + \frac{\gamma\|A_1x_1+A_2x_2^0-a\|^2}{2T} + \frac{\|H\| D_\cX^2  + \|G\| D_\cY^2}{2T} + \frac{\sum_{k = 0}^{T-1}\|u^k\|^2}{2\|G\|\cdot T}.\nonumber
	\end{eqnarray}
	By setting 
	$$\lambda = -\rho\cdot\frac{A\bar{\bx} - a}{\|A\bar{\bx} - a\|},\quad \by = \argmax_{\by'\in \cY}\,\, \Phi(\bar{\bx},\by'),\quad\mbox{and}\quad \bx = \argmin_{\bx'\in\cX,A\bx' = a} \Phi(\bx',\bar \by)
	$$
	and applying the fact that 
	$\lambda^0 = 0,$ $\|\lambda-\lambda^0\|^2 = \rho^2,$ 
	$\|u^k\|^2 \leq  \ell^2$ and $\|A_1x_1+A_2x_2^0-b\|^2 = \|A_2(x_2^0-x_2)\|^2\leq \|A_2\|^2\cdot D_{\cX_2}^2$ proves the theorem.  
\end{proof}

\subsection{The  SEG-ADMM for smooth problem}
Due to the nonsmoothness of $\Phi(\bx,\cdot)$, our SSG-ADMM algorithm applies a supergradient ascent step to update $\by$, resulting in an $\cO(1/\sqrt{T})$ convergence rate. In this section, we show that an improved $\cO(1/T)$ convergence can be obtained by replacing the supergradient step with an extragradient step, given better smoothness condition. Based on this feature, we call the new algorithm \underbar{S}addle-point \underbar{E}xtra\underbar{G}radient \underbar{A}lternating \underbar{D}irection \underbar{M}ethod of \underbar{M}ultipliers (SEG-ADMM), as is decribed by Algorithm \ref{alg:SEG-ADMM}. 

\begin{algorithm2e}[H]
	\DontPrintSemicolon  
	\caption{The  SEG-ADMM Algorithm}
	\label{alg:SEG-ADMM}
	\textbf{input}: $\mathbf{x}^0\in\cX, \by^0\in\cY$, $\lambda^0=0\in\R^n$, $\gamma>0$, and matrices $G\succ0$ and $H_i\succ 0$.\\
	\For{$k = 0,...,T$}{
		Apply the gradient ascent step:\vspace{-0.1cm} 
		$$\hat \by^{k+1} = \argmin_{\by\in\cY} \,\,\frac{1}{2}\left\|\by-[\by^k + G^{-1}\cdot \nabla_\by \Phi(\mathbf{x}^k,\by^k)]\right\|^2_G. \vspace{-0.1cm}$$\\
		Apply the \texttt{Prox-ADMM} submodule:\vspace{-0.1cm}
		$$(\mathbf{x}^{k+1},\lambda^{k+1}) = \texttt{Prox-ADMM}(\mathbf{x}^k,\lambda^k;\hat \by^{k+1};\gamma, \{H_i\}_{i=1}^N).\vspace{-0.1cm}$$\\
		Apply the extra-gradient ascent step:\vspace{-0.1cm}
		$$\by^{k+1} = \argmin_{\by\in\cY} \,\,\frac{1}{2}\left\|\by-[\by^k + G^{-1}\cdot \nabla_\by \Phi(\mathbf{x}^{k+1},\hat \by^{k+1})]\right\|^2_G. \vspace{-0.1cm}$$}
	\textbf{output:} $\bar x = \frac{1}{T}\sum_{k = 1}^T\mathbf{x}^k$, $\bar \by = \frac{1}{T}\sum_{k = 1}^{T}\hat \by^k$.
\end{algorithm2e} 

Similar to the analysis of SSG-ADMM, the $\bx$ update of SEG-ADMM is fully characterized by Lemma \ref{lemma:2-ADMM-step} by setting $\tilde{\by} = \hat{\by}^{k+1}$.  We only need to analyze the $\by$ update in the following lemma. See the proof in Appendix \ref{appdx:lemma:SEG-step}.

\begin{lemma}
	\label{lemma:SEG-step}
	Suppose Assumption \ref{assumption:Lipschitz-1sided} holds. Let $\by^k$, $\hat \by^{k+1}$ and $\by^{k+1}$ be generated by Algorithm \ref{alg:SEG-ADMM}. Then for $\forall\by\in\cY$, it holds that 
	\begin{eqnarray}
		\label{lm:SEG-step}
		\Phi(\bx^{k+1},\by) - \Phi(\bx^{k+1},\hat \by^{k+1})
		& \leq & \frac{1}{2}\left(\|\by-\by^k\|^2_G - \|\by-\by^{k+1}\|^2_G\right) + \frac{L_\by}{2}\|\bx^k-\bx^{k+1}\|^2 \nonumber \\
		&  &  - \frac{1}{2}\big(\|\hat \by^{k+1}-\by^{k}\|^2_{G-L_\by I} +\|\hat \by^{k+1}-\by^{k+1}\|^2_{G-L_\by I}\big).
	\end{eqnarray}
\end{lemma}
It is worth noting that the error term $+\frac{L_\by}{2}\|\bx^k-\bx^{k+1}\|^2$ in \eqref{lm:SEG-step} will be canceled by the descent term $-\frac{1}{2}\|\bx^k-\bx^{k+1}\|^2_{H-L_\bx I}$ in \eqref{lm:2-ADMM-step}. This is the reason why choosing the proximal version of ADMM (\texttt{Prox-ADMM}) for $\bx$ update is essential.  With the Lemma \ref{lemma:2-ADMM-step} and Lemma \ref{lemma:SEG-step}, we can prove the following theorem.
\begin{theorem}[Convergence of SEG-ADMM]
	\label{theorem:complexity}
	Consider the problem \ref{prob:1-sided} with $N=2$. Suppose Assumptions \ref{assumption:Convexity-1sided} and \ref{assumption:Lipschitz-1sided} hold.  Suppose the output $(\bar{\bx},\bar \by)$ is returned by Algorithm \ref{alg:SEG-ADMM} after $T$ iterations. 
	As long as we choose $H_i \succeq  (L_\bx+L_\by)\cdot I, i=1,\cdots,N,$ and $G \succeq  L_\by\cdot I$, it holds for $\forall \rho>0$ that 
	\begin{eqnarray} 
		\max_{\by\in \cY} \Phi(\bar{\bx} ,\by) \!-\!\! \min_{\overset{\bx\in\cX}{A\bx= a}}\Phi(\bx,\bar \by)\!  + \! \rho\|A\bar{\bx} - a\|\leq \frac{\rho^2/\gamma + \gamma\|A_2\|^2 D_{\cX_2}^2 + \|H\| D_{\cX}^2+\|G\| D_{\cY}^2}{2T}.\nonumber
	\end{eqnarray}
	By Lemma \ref{lemma:opt-cond}, it takes $T=\cO(\epsilon^{-1})$ iterations to reach an $\epsilon$-saddle point. 
\end{theorem}
\noindent The proof of this theorem is similar to that of Theorem \ref{theorem:complexity-ssg}, thus we omit it here. 

\subsection{Discussion} 
Given the  analysis of the SSG-ADMM and SEG-ADMM algorithms, we can see that the traditional ADMM algorithm can be easily extended to the multi-block affinely constrained minimax problems of the form \eqref{prob:1-sided}, by applying a proximal ADMM update on the affinely constrained variable $\bx$, while making a simple supergradient update for $\by$ or an extragradient update for $\by$ that sandwiches the $\bx$ update. Therefore, we would like to make a brief discussion on the pros and cons of this approach.\vspace{0.2cm}

\noindent{\bf Versatility of analysis.}
Throughout Section \ref{section:1side-2block}, and the analysis of the SEG-ADMM with $N\geq3$ in Appendix \ref{appdx:Ngeq3}, it can be observed that the analysis of $\bx$ blocks' \texttt{Prox-ADMM} step and the analysis of $\by$ block's supergradient/extragradient step are almost independent. This  implies that many existing results of ADMM under different conditions can be incorporated into the analysis of the $\bx$ update in  SSG-ADMM and SEG-ADMM, making the algorithm framework versatile for many potential extensions.\vspace{0.2cm}

\noindent{\bf Restriction in $N$.}
Like classical ADMM in convex optimization, the  SSG-ADMM  and   SEG-ADMM can easily diverge when $N\geq3$, unless additional assumptions or algorithm adaptation is available. For example, an $\cO(1/T)$ convergence can be derived for SEG-ADMM if we adopt the partial strong convexity condition for problem \eqref{prob:1-sided}. When such condition does not hold, we can adopt the perturbation strategy to obtain a worse $\cO(1/\sqrt{T})$ convergence rate. We place these results in Appendix \ref{appdx:Ngeq3}. \vspace{0.2cm}

\noindent{\bf Inability against general problem \eqref{prob:main}.} In contrast to the versatility of SSG-ADMM and SEG-ADMM in terms of the simpler problem \eqref{prob:1-sided}, the asymmetry of ADMM updates makes it very hard to be analyzed when we extend it to the general problem \eqref{prob:main}, where both $\bx$ and $\by$ have affinely constrained multi-block structure. This is also the reason why we propose the EGMM algorithm in the next section.

\section{The extra-gradient method of multipliers (EGMM)}
\label{section:two-sided}
As discussed in the last section, it is hard for the ADMM-type methods to handle the general problem \eqref{prob:main}. To resolve this difficulty, we introduce the \underline{E}xtra-\underline{G}radient \underline{M}ethod of \underline{M}ultipliers (EGMM), which takes a strategy that represents a significant departure from the ADMM-type methods in section \ref{section:1side-2block}. Note that problem \eqref{prob:main} can be compactly written as \eqref{prob:main-simple}, which is rewritten below for readers' convenience:
\begin{eqnarray} 
	\min_{\mathbf{x}\in\cX}\max_{\by\in\cY}\Phi(\bx,\by):= h(\bx) + \Psi(\bx,\by) - g(\by)\quad\mbox{s.t.}\quad A\bx = a, B\by= b.\nonumber
\end{eqnarray} 
For this problem, we make the following assumptions. 
\begin{assumption}
	\label{assumption:Convexity}
	$h_i(\cdot)$ and $g_j(\cdot)$ are convex and lower semi-continuous functions, for $i = 1,...,N$ and $j = 1,...,M$.
\end{assumption}

\begin{assumption}
	\label{assumption:Lipschitz-psi}
	We assume $\Psi(\cdot,\by)$ is  convex in $\cX$ for  $\forall\by\in\cY$ and $\Psi(\bx,\cdot)$ is concave in $\cY$ for $\forall\bx\in\cX$. The gradient of $\Psi$ is $L$-Lipschitz continuous, i.e., 
	$$\|\nabla \Psi(\bx,\by) - \nabla \Psi(\bx',\by')\|\leq L\sqrt{\|\bx-\bx'\|^2+\|\by-\by'\|^2},\quad\forall\bx,\bx'\in\cX, \,\,\forall \by,\by'\in\cY,$$
	where $\nabla \Psi(\bx,\by) = (\nabla_\bx \Psi(\bx,\by),\nabla \Psi_\by(\bx,\by))$ is the full gradient of the coupling term.
	The overall objective function $\Phi(\cdot,\cdot)$ is bounded over $\cX\times\cY$.
\end{assumption}
\noindent To derive the EGMM algorithm, let us first rewrite \eqref{prob:main-simple} as a new minimax problem:
\begin{eqnarray} 
	\label{prob:main-Lag-form}
	\mathop{\mathrm{min}}_{\overset{\mu\in\RR^m}{\mathbf{x}\in\cX}}\,\,\mathop{\mathrm{max}}_{\overset{\lambda\in\RR^n}{\by\in\cY}} \,\,\cL(\bx,\!\by;\!\lambda,\!\mu)\!:=\! h(\bx) \!+\! \Psi(\bx,\!\by) \!-\! g(\by) \!-\! \langle A\bx \!-\! a,\!\lambda\rangle \!+\!  \langle B\by \!-\! b,\!\mu\rangle.
\end{eqnarray}
Then our EGMM algorithm attempts to solve the original problem \eqref{prob:main-simple} by working on problem \eqref{prob:main-Lag-form}. However, we should also notice that, as long as $(\bar \bx,\bar \by;\bar{\lambda},\bar \mu)$ does not satisfy $A\bar{\bx} = a$ and $B\bar{\by}=b$ simultaneously, the duality gap of \eqref{prob:main-Lag-form} is infinity, i.e., 
$$\mathop{\mathrm{max}}_{\overset{\lambda\in\RR^n}{\by\in\cY}} \,\,\cL(\bar{\bx},\by;\bar \mu,\lambda) \,\,-\,\, \mathop{\mathrm{min}}_{\overset{\mu\in\RR^m}{\mathbf{x}\in\cX}} \,\,\cL({\bx},\bar\by; \mu,\bar\lambda)\,\, = \,\,+\infty.$$
Therefore, one should not simply apply the existing analysis of convex-concave saddle point problems and try to prove the convergence in terms of the duality gap. Instead, we should carefully utilize the special structure of the original problem, and try to analyze the convergence to an $\epsilon$-saddle point in the sense of Definition \ref{definition:eps}. If we denote
$z = \begin{bmatrix}
	\bx^\top,\by^\top,\lambda^\top,\mu^\top
\end{bmatrix}^\top\in\cZ:=\cX\times\cY\times\R^n\times\R^m$, $R(z):= h(\bx) + g(\by)$,
$$F(z) = \begin{bmatrix}
	\,\,\,\,\,\nabla_\bx \Psi(\bx,\by) - A^\top\lambda\\
	-\nabla_\by \Psi(\bx,\by) - B^\top\mu\\
	A\bx-a\\
	B\by-b
\end{bmatrix},\quad\mbox{and}\quad H = \begin{bmatrix}
	\sigma_\bx  I_{d_\bx} & & & \\
	& \sigma_\by  I_{d_\by} & & \\
	& & \sigma_\lambda I_n & \\
	& & & \sigma_\mu I_m
\end{bmatrix} $$
where $d_\bx = \sum_{i}d_x^i$, $d_\by = \sum_{j}d_y^j$, $I_d$ is a $d\times d$ identity matrix, and $\sigma_\bx,\sigma_\by,\sigma_\lambda,\sigma_\mu>0$ are some positive constants. We can simply write the EGMM algorithm as
\begin{equation}
	\label{algo:main-simplified}
	\begin{cases}\hat z^{k+1}  = \argmin_{z\in\cZ} \frac{1}{2}\left\|z - [z^k - H^{-1}\cdot F(z^k)]\right\|^2_{H} + R(z),\vspace{0.1cm}\\
		z^{k+1}  = \argmin_{z\in\cZ} \frac{1}{2}\left\|z - [z^k - H^{-1}\cdot F(\hat z^{k+1})]\right\|^2_H + R(z).
	\end{cases}
\end{equation}
Without the augmented Lagrangian terms that couple the multiple blocks of $\bx$ and $\by$ together, i.e. $\|\sum_i A_ix_i-a\|^2$ and $\|\sum_j B_jy_j-b\|^2$, the above subproblem is a group of separable small subproblems. Assuming the solvability of  these subproblems,  we describe the EGMM method as Algorithm \ref{alg:main}. 

%For EGMM algorithm, we have the following lemma. 
\begin{lemma}
	\label{lemma:EG}
	Let $z^k$, $\hat z^{k+1}$ and $z^{k+1}$ be generated by Algorithm \ref{alg:main}. Then  
	\begin{align}
		\label{lm:EG}
		&\quad\,\, R(\hat z^{k+1}) - R(z) + \big\langle F(\hat z^{k+1}),\hat z^{k+1} - z\big\rangle\\
		&\leq \frac{1}{2}\|z^k - z\|^2_H  - \frac{1}{2}\|z - z^{k+1}\|^2_H
		- \frac{1}{2}\|z^k - \hat z^{k+1}\|^2_{H-G} -\frac{1}{2}\|\hat z^{k+1} - z^{k+1}\|^2_{H-G}\nonumber
	\end{align}
	for $\forall z\!\in\!\cZ$, where $G = \mathrm{Diag}\left\{\frac{L+\|A\|}{2}I_{\bx},\frac{L+\|B\|}{2}I_\by,\frac{\|A\|}{2}I_n,\frac{\|B\|}{2}I_m\right\}$ is a diagonal matrix.
\end{lemma} 
\begin{proof} 
	Consider the compact form of Algorithm \ref{alg:main} summarized in \eqref{algo:main-simplified}. The KKT condition of the two subproblems implies that $\exists \hat v^{k+1}\in\partial R(\hat z^{k+1})$ and $v^{k+1}\in\partial R(z ^{k+1})$ s.t. 
	\begin{equation}
		\label{thm:pf-opt-0.5}
		\left\langle H(\hat z^{k+1} - z^k) + F(z^k) + \hat v^{k+1}, z-\hat z^{k+1}\right\rangle \geq 0, \quad\forall z\in\cZ
	\end{equation}
	\begin{equation}
		\label{thm:pf-opt-1}
		\left\langle H( z^{k+1} - z^k) + F(\hat z^{k+1}) + v^{k+1}, z - z^{k+1}\right\rangle \geq 0, \quad\forall z\in\cZ
	\end{equation}
	Set $z = z^{k+1}$ in \eqref{thm:pf-opt-0.5} yields
	\begin{eqnarray} 
		\big\langle F(z^k) + \hat v^{k+1}, \hat z^{k+1}-z^{k+1}\big\rangle & \leq & \left\langle H(\hat z^{k+1} - z^k), z^{k+1}-\hat z^{k+1}\right\rangle \nonumber\\
		& = & \frac{1}{2}\|z^k - z^{k+1}\|^2_H - \frac{1}{2}\|z^k - \hat z^{k+1}\|^2_H - \frac{1}{2}\|\hat z^{k+1} - z^{k+1}\|^2_H\nonumber.
	\end{eqnarray}
	Similarly, \eqref{thm:pf-opt-1} can be equivalently written as 
	\begin{eqnarray} 
		\big\langle F(\hat z^{k+1}) + v^{k+1}, z^{k+1}-z\big\rangle & \leq & \frac{1}{2}\|z^k - z\|^2_H - \frac{1}{2}\|z^k - z^{k+1}\|^2_H - \frac{1}{2}\|z - z^{k+1}\|^2_H\nonumber.
	\end{eqnarray}
	Summing up the above two inequalities yields:
	\begin{eqnarray}
		\label{thm:pf-2}
		&&\frac{1}{2}\|z^k - z\|^2_H  - \frac{1}{2}\|z - z^{k+1}\|^2_H- \frac{1}{2}\|z^k - \hat z^{k+1}\|^2_H - \frac{1}{2}\|\hat z^{k+1} - z^{k+1}\|^2_H \\
		& \geq & \big\langle F(\hat z^{k+1}) + v^{k+1}, z^{k+1}-z\big\rangle + \big\langle F(z^k) + \hat v^{k+1}, \hat z^{k+1}-z^{k+1}\big\rangle\nonumber\\
		& \overset{(i)}{\geq} & \big\langle F(\hat z^{k+1}), z^{k+1}-z\big\rangle + \big\langle F(z^k), \hat z^{k+1}-z^{k+1}\big\rangle + R(\hat z^{k+1}) - R(z)\nonumber\\
		& = & R(\hat z^{k+1}) - R(z) + \big\langle F(\hat z^{k+1}),\hat z^{k+1} - z\big\rangle + \big\langle F(\hat z^{k+1})-F(z^{k}), z^{k+1} - \hat z^{k+1}\big\rangle\nonumber\\
		& \overset{(ii)}{\geq} & R(\hat z^{k+1}) - R(z) + \big\langle F(\hat z^{k+1}),\hat z^{k+1} - z\big\rangle -\frac{1}{2}\|z^{k+1} - \hat z^{k+1}\|^2_G - \frac{1}{2}\|z^k - \hat z^{k+1}\|^2_G\nonumber
	\end{eqnarray}
	where (i) is due to the convexity of $R$:
	\begin{eqnarray}
		\big\langle  v^{k+1}, z^{k+1}-z\big\rangle + \big\langle   \hat v^{k+1}, \hat z^{k+1}-z^{k+1}\big\rangle &\geq& R(z^{k+1})-R(z) + R(\hat z^{k+1}) - R(z^{k+1})\nonumber\\
		& = & R(\hat z^{k+1}) - R(z)\nonumber
	\end{eqnarray}
	and (ii) is due to the Lipschitz continuity of $\nabla\Psi$ and Cauchy inequality:
	\begin{eqnarray}
		&&\big\langle F(\hat z^{k+1})-F(z^{k}),z^{k+1} - \hat z^{k+1}\big\rangle \nonumber\\
		&=& \big\langle \nabla_\bx\Psi(\hat \bx^{k+1},\hat \by^{k+1}) - \nabla_\bx\Psi(\bx^{k},\by^{k}) - A^\top(\hat\lambda^{k+1}-\lambda^{k}),\bx^{k+1} - \hat \bx^{k+1}\big\rangle\nonumber\\
		& & \!\!\!\!\!+ \big\langle \nabla_\by\Psi(\bx^{k},\by^{k}) - \nabla_\by\Psi(\hat \bx^{k+1},\hat \by^{k+1}) - B^\top(\hat\mu^{k+1}-\mu^{k}),\by^{k+1} - \hat \by^{k+1}\big\rangle\nonumber\\
		&& \!\!\!\!\!+ \big\langle A(\hat\bx^{k+1} - \bx^{k}),\lambda^{k+1} - \hat{\lambda}^{k+1}\big\rangle + \big\langle B(\hat\by^{k+1} - \by^{k}),\mu^{k+1} - \hat{\mu}^{k+1}\big\rangle\nonumber\\
		& \geq & -\frac{1}{2L}\|\nabla_\bx\Psi(\hat \bx^{k+1},\hat \by^{k+1}) - \nabla_\bx\Psi(\bx^{k},\by^{k})\|^2-\frac{L+\|A\|}{2}\|\hat{\bx}^{k+1} - \bx^{k+1}\|^2 -  \frac{\|A\|}{2}\|\hat\lambda^{k+1}-\lambda^{k}\|^2 \nonumber\\
		& & -\frac{1}{2L}\|\nabla_\by\Psi(\hat \bx^{k+1},\hat \by^{k+1}) - \nabla_\by\Psi(\bx^{k},\by^{k})\|^2-\frac{L+\|B\|}{2}\|\hat{\by}^{k+1} - \by^{k+1}\|^2 -  \frac{\|B\|}{2}\|\hat\mu^{k+1}-\mu^{k}\|^2\nonumber\\
		& & - \frac{\|A\|}{2}\|\hat\lambda^{k+1}-\lambda^{k+1}\|^2- \frac{\|B\|}{2}\|\hat\mu^{k+1}-\mu^{k+1}\|^2-\frac{\|A\|}{2}\|\hat{\bx}^{k+1} - \bx^{k}\|^2-\frac{\|B\|}{2}\|\hat{\by}^{k+1} - \by^{k}\|^2\nonumber\\
		&\geq& -\frac{1}{2}\|z^{k+1} - \hat z^{k+1}\|^2_G - \frac{1}{2}\|z^k - \hat z^{k+1}\|^2_G.\nonumber
	\end{eqnarray} 
	Hence we complete the proof. 
\end{proof}
Until this point, the analysis of Lemma \ref{lemma:EG} follows that of the extragradient (EG) method, except for additional effort to deal with the positive definite scaling matrix $H$ and proxiaml-gradient step of \eqref{algo:main-simplified} instead of the simple gradient step of EG. In the next step, our analysis diverges from the existing methods by skipping the convergence analysis of the duality gap as well as that of the multiplier sequence $\{\hat \lambda^k,\hat\mu^k\}_{k=1,2,...}$. Instead, we utilize the structure of the original problem \eqref{prob:main} and only focus on the convergence of $\{\hat\bx^k,\hat\by_k\}_{k=1,2,...}$. We present the result in the next theorem.\vspace{0.2cm}

\begin{algorithm2e}
	\DontPrintSemicolon  
	\caption{The  EGMM algorithm}
	\label{alg:main}
	\textbf{input}: ${\mathbf{x}}^0\in\cX$, $\by^0\in\cY$,  $\lambda^0=0\in\R^n$, $\mu^0=0\in\R^m$,  matrix $H$.\\
	\For{$k=0,...,T-1$}{
		Update the variable blocks by one proximal gradient step:\\
		\For{$i = 1,2,...,N$, $j = 1,2,...,M$}{
			\vspace{0.1cm} Denote $\begin{cases}F_{x_i} := (\nabla_{\bx_i}\Psi(\bx^k,\by^k) - A_i^\top\lambda^k\\
				F_{y_j} := \nabla_{\by_j}\Psi(\bx^k,\by^k) - B_j^\top\mu^k
			\end{cases}$ $\!\!\!\!\!\!$for $\forall i,j$. Then compute\vspace{-0.1cm}
			\begin{equation}
				\label{eqn:alg-x0.5}
				\hat x_i^{k+1} = \argmin_{x_i\in\cX_i} \frac{\sigma_\bx}{2}\left\|x_i - [x_i^k - \sigma_\bx^{-1}\cdot F_{x_i}] \right\|^2 + h_i(x_i) \vspace{-0.4cm}
			\end{equation}
			\begin{equation}
				\label{eqn:alg-y0.5}
				\hat y_j^{k+1} = \argmin_{y_j\in\cY_j} \frac{\sigma_\by}{2}\left\|y_j - [y_i^k + \sigma_\by^{-1}\cdot F_{y_j}] \right\|^2 + g_j(y_j)\vspace{-0.4cm}
			\end{equation} 
		}
		\vspace{0.1cm}Update the multipliers by one gradient 
		\begin{equation}
			\label{eqn:alg-mul0.5}
			\begin{cases}
				\hat \lambda^{k+1} = \lambda^k - \sigma_\lambda^{-1}\cdot\big(\sum_{i=1}^NA_ix_i^k - a\big)\\
				\hat\mu^{k+1} = \mu^k - \sigma_\mu^{-1}\cdot\big(\sum_{j = 1}^MB_jy_j^k - b\big)
			\end{cases}
		\end{equation}\\
		Update the variable blocks by one proximal gradient step:\vspace{0.1cm} \\
		\For{$i = 1,2,...,N$, $j = 1,2,...,M$}{
			\vspace{0.1cm} Denote $\begin{cases}
				\hat F_{x_i} := \nabla_{\bx_i}\Psi(\hat\bx^{k+1},\hat\by^{k+1}) - A_i^\top\hat\lambda^{k+1}\\
				F_{y_j}:=\nabla_{\by_j}\Psi(\hat\bx^{k+1}\!,\hat\by^{k+1}) - B_j^\top\hat\mu^{k+1}
			\end{cases}$$\!\!\!\!\!$for $\forall i,j$. Then compute \vspace{-0.1cm}
			\begin{equation}
				\label{eqn:alg-x1}
				x_i^{k+1} = \argmin_{x_i\in\cX_i} \frac{\sigma_\bx}{2} \big\|x_i \!-\! [x_i^k \!-\! \sigma_\bx^{-1}\!\cdot\hat F_{x_i}]\big\|^2 \!\!+\! h_i(x_i)\vspace{-0.4cm}
			\end{equation}
			\begin{equation}
				\label{eqn:alg-y1}
				\!y_j^{k+1} \!= \argmin_{y_j\in\cY_j} \frac{\sigma_\by}{2}\!\left\|y_j \!-\! [y_i^k \!+\! \sigma_\by^{-1}\!\cdot \hat F_{y_j}] \right\|^2 \!\!+\! g_j(y_j)\vspace{-0.4cm}
			\end{equation} 
		}
		\vspace{0.1cm}Update the multipliers  by one gradient step:
		\begin{equation}
			\label{eqn:alg-mul1}
			\begin{cases}
				\lambda^{k+1} = \lambda^k - \sigma_\lambda^{-1}\cdot\big(\sum_{i=1}^NA_i\hat x_i^{k+1} - a\big)\\
				\mu^{k+1} = \mu^k - \sigma_\mu^{-1}\cdot\big(\sum_{j = 1}^MB_j\hat y_j^{k+1} - b\big)
			\end{cases} \vspace{-0.4cm}
	\end{equation}}
	{\bf output} $\bar{\bx } = \frac{1}{T}\sum_{k = 1}^T\hat{\bx}^k$ and $\bar{\by } = \frac{1}{T}\sum_{k = 1}^T\hat{\by}^k$.\vspace{-0.0cm}
\end{algorithm2e}  

\begin{theorem}[Convergence of EGMM]
	\label{theorem:EGMM}
	For problem \eqref{prob:main} with general block numbers $N,M\geq1$, suppose Assumption \ref{assumption:Convexity} and \ref{assumption:Lipschitz-psi} hold. 
	Let $(\bar \bx,\bar{\by})$ be generated by Algorithm \ref{alg:main} after $T$ iterations, with $\sigma_\bx = \frac{L+\|A\|}{2}, \sigma_\by = \frac{L+\|B\|}{2},\sigma_\lambda = \frac{\|A\|}{2},\sigma_\mu = \frac{\|B\|}{2}$. 
	Then it holds for any $\rho>0$ that \vspace{-0.1cm}
	\begin{eqnarray}
		\max_{\overset{\by\in\cY}{B\by = b}}\Phi(\bar \bx,\by) - \min_{\overset{\bx\in\cX}{A\bx = a}}\Phi(\bx,\bar \by)\!&\!+\!&\!\rho\|A\bar\bx-a\| + \rho\|B\bar\by-b\|\nonumber\\
		&\leq&\frac{L+\|A\|}{ 4T}\cdot D_\cX^2 + \frac{L+\|B\|}{ 4T}\cdot D_\cY^2 + \frac{\|A\|+\|B\|}{2T}\cdot \rho^2.\nonumber
	\end{eqnarray}
	Therefore, it takes $T=\cO(\epsilon^{-1})$ iterations to reach an $\epsilon$-saddle point. 
\end{theorem}
\begin{proof}
	Consider the LHS of \eqref{lm:EG}. By utilizing the detailed form of $F(z)$, we have 
	\begin{eqnarray}
		\label{thm:main-1}
		&&R(\hat z^{k+1}) - R(z) + \big\langle F(\hat z^{k+1}),\hat z^{k+1} - z\big\rangle\\
		& = & R(\hat z^{k+1}) - R(z) + \big\langle \nabla_\bx\Psi(\hat \bx^{k+1},\hat \by^{k+1}) - A^\top\hat{\lambda}^{k+1},\hat \bx^{k+1} - \bx\big\rangle  \nonumber\\
		&& + \big\langle A\hat\bx^{k+1}-a, \hat \lambda^{k+1}-\lambda\big\rangle + \big\langle B\hat\by^{k+1}-b, \hat \mu^{k+1}-\mu\big\rangle \nonumber\\
		&&+ \big\langle -\nabla_\by\Psi(\hat \bx^{k+1},\hat \by^{k+1}) - B^\top\hat{\mu}^{k+1},\hat \by^{k+1} - \by\big\rangle \nonumber\\
		&\overset{(i)}{\geq}& \Phi(\hat \bx^{k+1},\by) - \Phi(\bx,\hat \by^{k+1}) - \langle A^\top\hat{\lambda}^{k+1},\hat \bx^{k+1} - \bx\big\rangle  + \langle A\hat\bx^{k+1}-a, \hat \lambda^{k+1}-\lambda\rangle\nonumber\\
		&& - \langle B^\top\hat{\mu}^{k+1},\hat \by^{k+1} - \by\rangle  + \langle B\hat\by^{k+1}-b, \hat \mu^{k+1}-\mu\rangle\nonumber
	\end{eqnarray}
	where (i) is because  of the definition of $R$ and the convex-concave nature of $\Psi$:
	$$R(\hat z^{k+1}) - R(z) = \big(h(\hat\bx^{k+1}) - g(\by)\big) - \big(h(\bx) - g(\hat{\by}^{k+1})\big)$$
	$$\big\langle \nabla_\bx\Psi(\hat \bx^{k+1},\hat \by^{k+1}),\hat \bx^{k+1} - \bx\big\rangle \geq \Psi(\hat \bx^{k+1},\hat \by^{k+1}) - \Psi(\bx,\hat \by^{k+1})$$
	$$\big\langle -\nabla_\by\Psi(\hat \bx^{k+1},\hat \by^{k+1}),\hat \by^{k+1} - \by\big\rangle\geq \Psi(\hat \bx^{k+1},\by) - \Psi(\hat \bx^{k+1},\hat \by^{k+1})$$
	If we further require $(\bx,\by)\in\cX\times\cY$ to satisfy $A\bx = a$ and $B\by= b$, then we have
	\begin{eqnarray}
		&&\big\langle - A^\top\hat{\lambda}^{k+1},\hat \bx^{k+1} - \bx\big\rangle  + \big\langle A\hat\bx^{k+1}-a, \hat \lambda^{k+1}-\lambda\big\rangle\nonumber\\
		& = & \big\langle - \hat{\lambda}^{k+1},A\hat \bx^{k+1} - A\bx\big\rangle  + \big\langle A\hat\bx^{k+1}-a, \hat \lambda^{k+1}-\lambda\big\rangle\nonumber\\
		& = & -\big\langle A\hat\bx^{k+1}-a, \lambda\big\rangle\nonumber
	\end{eqnarray}
	and 
	\begin{eqnarray}
		&&\big\langle - B^\top\hat{\mu}^{k+1},\hat \by^{k+1} - \by\big\rangle  + \big\langle B\hat\by^{k+1}-b, \hat \mu^{k+1}-\mu\big\rangle\nonumber\\
		& = & \big\langle - \hat{\mu}^{k+1},B\hat \by^{k+1} - B\by\big\rangle  + \big\langle B\hat\by^{k+1}-b, \hat \mu^{k+1}-\mu\big\rangle\nonumber\\
		& = & - \big\langle B\hat\by^{k+1}-b, \mu\big\rangle\nonumber.
	\end{eqnarray}
	Combined with \eqref{thm:main-1} and Lemma \ref{lemma:EG}, we have 
	\begin{eqnarray}
		\label{thm:main-2}
		&&\Phi(\hat \bx^{k+1},\by) - \Phi(\bx,\hat \by^{k+1})-\big\langle A\hat\bx^{k+1}-a, \lambda\big\rangle- \big\langle B\hat\by^{k+1}-b, \mu\big\rangle\nonumber\\
		& \leq & R(\hat z^{k+1}) - R(z) + \big\langle F(\hat z^{k+1}),\hat z^{k+1} - z\big\rangle\\
		& \leq & \frac{1}{2}\|z^k - z\|^2_H  - \frac{1}{\eta}\|z - z^{k+1}\|^2_H\nonumber
	\end{eqnarray}
	where the $-\|\cdot\|_{H-G}^2$ terms in Lemma \ref{lemma:EG} vanish since $H=G$. Next, similar to the analysis of Theorem \ref{theorem:complexity-ssg}, we can average \eqref{thm:main-2} and use Jensen's inequality to yield:
	\begin{eqnarray} 
		\Phi(\bar \bx,\by) - \Phi(\bx,\bar \by)-\big\langle A\bar\bx-a, \lambda\big\rangle- \big\langle B\bar\by-b, \mu\big\rangle 
		\leq \frac{\|z-z^0\|^2_H}{2\cdot T}\nonumber. 
	\end{eqnarray}
	Since the above inequality holds for $\forall(\bx,\by)\in\cX\times\cY$ with $A\bx=a,B\by=b$, we can set 
	$\lambda = -\rho\cdot\frac{A\bar{\bx}-a}{\|A\bar{\bx}-a\|}$ and  $\mu = -\rho\cdot\frac{B\bar{\by}-b}{\|B\bar{\by}-b\|}$, also notice that $\mu^0 = 0, \lambda^0 = 0$, we have 
	\begin{align}
		\max_{\overset{\by\in\cY}{B\by=b}}\Phi(\bar \bx,\by) - \min_{\overset{\bx\in\cX}{A\bx=a}}\Phi(\bx,\bar \by)&+\rho\|A\bar\bx-a\| + \rho\|B\bar\by-b\|\nonumber\\
		&\leq\frac{L+\|A\|}{ 4T}\cdot D_\cX^2 + \frac{L+\|B\|}{ 4T}\cdot D_\cY^2 + \frac{\|A\|+\|B\|}{2T}\cdot \rho^2\nonumber,
	\end{align}
	which proves the theorem. 
\end{proof}
It can be noticed that although the duality gap w.r.t. problem \eqref{prob:main-Lag-form} is always infinity, the special structure of the original problem \eqref{prob:main} allows us to circumvent this issue and establish a convergence to the $\epsilon$-saddle point in sense of Definition \ref{definition:eps}, which is different from the traditional duality measure.

\subsection{Extension to conic inequality constrained problem}
The  EGMM algorithm can also easily adapt to the conic inequality constrained problems:
\begin{eqnarray} 
	\label{prob:conic}
	\min_{\mathbf{x}\in\cX}\max_{\by\in\cY}\Phi(\bx,\by):= h(\bx) + \Psi(\bx,\by) - g(\by)\quad\mbox{s.t.}\quad A\mathbf{x} \preceq_{K_1} a, \,\,\,B\by \preceq_{K_2} b,
\end{eqnarray} 
where $K_1,K_2$ are two closed convex cones. The notation $a\preceq_K b$ means that $b-a\in K$. Therefore, we can add two slack variables $x_{N+1}$ and $y_{M+1}$ and reformulate \eqref{prob:conic} as 
\begin{eqnarray*}  
	&&\mathop{\mathrm{minimize}}_{\mathbf{x}\in\cX,x_{N+1}\in\cX_{N+1}}\,\mathop{\mathrm{maximize}}_{\by\in\cY,y_{M+1}\in\cY_{M+1}}\Phi(\bx,\by):= h(\bx) + \Psi(\bx,\by) - g(\by)\\
	&&\mbox{s.t.}\qquad A\mathbf{x}  + x_{N+1} =  a, \,\,\,B\by + y_{M+1} = b, \nonumber\\
	&&\qquad\quad\,\,\cX_{N+1} := \{w\in\RR^n: w\in K_1, \|w\|\leq \|a\| + \|A\|D_\cX\}\nonumber\\
	&& \qquad\quad\,\,\cY_{M+1} := \{w\in\RR^m: w\in K_2, \|w\|\leq \|b\| + \|B\|D_\cY\}\nonumber
\end{eqnarray*} 
which is a special case of \eqref{prob:main}. Note that  EGMM works regardless of the number of blocks, we can again apply this algorithm to yield an $\cO(1/T)$ convergence rate.

\section{Primal-dual vs. approximate dual: a comparison between EGMM and ADMM}
\label{section:discussion}
Throughout the previous discussion, it is observed that the EGMM, SSG-ADMM and SEG-ADMM algorithms share the same advantage of subproblem separability, while EGMM has much better theoretical guarantees w.r.t. the general problem \eqref{prob:main} under various scenarios. In this section, we would like to discuss the insights of such advantage by inspecting EGMM and ADMM under the classical affinely constrained multi-block convex optimization problem \eqref{prob:admm}.

\subsection{EGMM \& ADMM for minimization}
Consider the special case of our main problem \eqref{prob:main} where $\Phi(\bx,\cdot)$ is a constant for $\forall \bx$. Then \eqref{prob:main} becomes the extensively studied multi-block convex optimization problem with affine constraint \eqref{prob:admm}, which can be written as 
\begin{equation*} 
	\min_{\bx\in\cX} \,\, \Phi(\bx) := h(\bx) + \Psi(\bx) \quad\mbox{s.t.}\quad A\bx = a.
\end{equation*}
The optimal solution of this problem is denoted as $\bx^*$. 
To formalize the discussion, we specialize the previous assumptions as follows. 

\begin{assumption} 
	\label{assumption:admm}
	For $i = 1,...,N$, $h_i(\cdot)$ is convex in $\cX_i$ and its proximal mapping can be efficiently evaluated. $\Psi(\cdot)$ is a smooth and convex in $\cX$, and $\nabla_\bx \Psi(\cdot)$ is $L_\bx$-Lipschitz continuous, i.e., $\|\nabla_\bx \Psi(\bx)-\nabla_\bx \Psi(\bx')\|\leq L_\bx\|\bx-\bx'\|$, $\forall x,x'\in\cX.$
\end{assumption}

\vspace{0.2cm}
\noindent\textbf{Multi-block ADMM algorithm.}$~$
For problem \eqref{prob:admm}, both SSG-ADMM and SEG-ADMM reduce to the proximal ADMM algorithm:
\begin{equation}
	\label{alg:ADMM}
	(\bx^{k+1},\lambda^{k+1}) = \texttt{Prox-ADMM}(\bx^k, \lambda^k; \by = {\bf null}; \gamma,\sigma).
\end{equation} 
For \eqref{alg:ADMM}, both the existing ADMM literature and Theorem \ref{theorem:complexity-ssg}, \ref{theorem:complexity} indicate the following result: Under Assumption \ref{assumption:admm} and $N = 2$. Set $\gamma>0$, $\sigma = L_\bx$. Then after $T$ iterations of \eqref{alg:ADMM}, we have for $\bar{\bx} = \frac{1}{T}\sum_{k = 1}^T\bx^k$ and $\forall\rho>0$ that
\begin{align*}
	\Phi(\bar {\bx}) - \Phi(\bx^*)  + \rho\|A\bar{\bx} - a\|
	\leq \cO\left(\frac{\rho^2/\gamma+ L_\bx D_{\cX}^2 + \gamma\|A_2\|^2 D_{\cX_2}^2}{2T} \right).
\end{align*}
When $N\geq3$, \eqref{alg:ADMM} diverges unless additional conditions or proper algorithm adaptation is made. E.g., making a strongly convex $\epsilon$-perturbation to problem \eqref{prob:admm}, see \eqref{prob:perturb}, yields an $\cO(1/\sqrt{T})$ convergence rate.

\vspace{0.2cm}
\noindent\textbf{EGMM algorithm.}$~$
For the optimization problem \eqref{prob:admm}, if we denote 
$$z = \begin{bmatrix}
	x\\ \lambda
\end{bmatrix}\!\!\in\!\cX\times\R^n, \quad F(z) = 
\begin{bmatrix}
	\nabla_\bx \Psi(\bx,\by) \!-\! A^\top\lambda\\
	A\bx-a
\end{bmatrix},\quad\mbox{and}\quad H = \begin{bmatrix}
	\frac{L_\bx + \|A\|}{2}  I_{d_\bx} & \\
	& \frac{\|A\|}{2} I_n 
\end{bmatrix} .$$
and $R(z):= h(\bx)$, then EGMM still takes the form of \eqref{algo:main-simplified}. As a corollary of Theorem \ref{theorem:EGMM}, it takes $T = \cO\Big(\frac{\|A\|\rho^2+(L_\bx + \|A\|) D_{\cX}^2}{\epsilon}\Big)$ iterations for EGMM to get an $\cO(\epsilon)$-optimal solution.

\subsection{Primal-dual method vs. approximate dual method}
To explain the restriction of the ADMM compared to EGMM, let us investigate the hidden logic behind these algorithms. 

\vspace{0.2cm}
\noindent\textbf{ADMM as an approximate dual method.} $~$ The traditional interpretation of ADMM type algorithms is to view them as an approximate gradient ascent of the dual function:
$$P_\gamma(\lambda) = \min_{\bx\in\cX} \,\,\cL_{\gamma}(\bx;\lambda):=h(\bx) + \Psi(\bx) -  \langle A\bx-a,\lambda\rangle + \frac{\gamma}{2}\|A\bx-a\|^2.$$
If we define $\bx^*(\lambda):=\argmin_{\bx\in\cX} \,\,\cL_{\gamma}(\bx;\lambda)$, then Danskin's theorem indicates that 
$$\nabla_{\!\lambda}P_\gamma(\lambda) = - (A\bx^*(\lambda)-a),$$
which is known to be $(1/\gamma)$-Lipschitz continuous in $\lambda$. Consider the special case where $\Psi(\bx)\equiv0$, the ADMM update in \eqref{alg:ADMM} with $\sigma = 0$ can be viewed as computing $\bx^{k+1}\approx \bx^*(\lambda^k)$ by one iteration of \emph{block coordinate minimization} in $\cL_{\gamma}(\bx;\lambda)$, and then update $\lambda^{k+1} = \lambda^k - \gamma\cdot(A\bx^{k+1}-a)\approx\lambda^k + \gamma\nabla_{\!\lambda}P_\gamma(\lambda^k)$. When $N=1$, $\bx^{k+1}=\bx^*(\lambda^k)$ holds exactly. When $N\geq3$, the approximation $\bx^{k+1}\approx \bx^*(\lambda^k)$ is very inaccurate and may cause divergence in hard instances. Therefore, instead of a primal-dual method, it is more appropriate to view ADMM as an \emph{approximate dual} method.

\vspace{0.2cm}
\noindent\textbf{EGMM as a primal-dual method.}$~$ In contrast to ADMM, EGMM proposed in this paper mainly focuses on the minimax formulation of problem \eqref{prob:admm}:
$$\min_{\bx\in\cX}\max_\lambda\,\, h(\bx) + \Psi(\bx) -  \langle A\bx-a,\lambda\rangle.$$
Note that EGMM can actually be viewed as a proximal gradient variant of the extragradient (EG) method, while adopting a positive definite gradient scaling, as well as a different convergence analysis. Therefore, EGMM experiences no issue of inexact dual gradient that has long troubled ADMM. Besides, the symmetricity of the primal-dual update makes it straightforward to be applied to the main problem \eqref{prob:main}. We can view EGMM as a fully \emph{primal-dual} method, while avoiding all convergence difficulties, it preserves the benefit of ADMM in solving small separable subproblems.

\section{Numerical Experiments on Team work RL} 
In this section, we consider the teamwork RL problem introduced in  \eqref{prob:mdp}. Given any partition of the state space $\cS = \cS_1\cup\cS_2\cup\cdots\cup\cS_n$, we consider the following general utility for this MDP: 
\begin{align*}
	&\rho_i(\mu(\cS_i,:)) \!:=\!\big\langle r(\cS_i,:),\mu(\cS_i,:)\big\rangle \!-\! \frac{\beta}{|\cS_i|}\!\sum_{s\in\cS_i}\!\!\left(\!\!\big\langle r(s,:),\mu(s,:)\big\rangle \!-\! \frac{\big\langle r(\cS_i,:),\mu(\cS_i,:)\big\rangle}{|\cS_i|}\!\right)^{\!\!2}.\\
	&\rho(\mu):= \mathop{\mathrm{min}}_{1\leq i\leq n}\,\,\Big\{\rho_i(\mu(\cS_i,:))\Big\}.
\end{align*}

%$$\rho(\mu):= \mathop{\mathrm{min}}_{1\leq i\leq n}\,\,\Big\{\rho_i(\mu(\cS_i,:))\Big\}$$
%$$\! \mathop{\mathrm{min}}_{1\leq i\leq n}\left\{\!\!  ss\right\}.$$
In this utility, for any $s\in\cS,a\in\cA$, $r(s,a)$ stands for the reward that node $s$ receives if it takes the action $a$ when visited by the system. Suppose $\mu$ is a state action occupancy measure under some policy $\pi$. Then the first term of $\rho_i(\mu(\cS_i,:))$  equals the total discounted cumulative reward received by the cluster $\cS_i$. The second term of $\rho(\mu(\cS_i,:))$, if we ignore the $-\beta$ factor, equals the variance among the cumulative rewards received by the different nodes in the cluster $\cS_i$. That is, the agent in charge of the cluster $\cS_i$ would like to maximize the overall reward of the cluster while using a variance penalty to impose fairness among the member nodes. For the whole system, the common goal is to maximize the minimum utility among the $n$ clusters. To solve this team RL with general utility, we reformulate it as follows
\begin{eqnarray}
	\label{prob:mdp-re}
	&&\mathop{\mathrm{maximize}}_{{\bf0}_{{|\cS|\!\times\!|\cA|}}\leq\mu\leq\frac{{\bf1}_{{|\cS|\!\times|\!\cA|}}}{1-\gamma}}\,\,\,\mathop{\mathrm{minimize}}_{y\geq{\bf0}_n,{\bf1}_n^\top y=1}\,\,\,\, \Phi(\mu,y) := \sum_{i=1}^n y_i\cdot \rho_i(\mu(\cS_i,:)) \\
	&&\mathrm{s.t.}\quad \sum_{a\in\cA}\mu(s,a) = \gamma\sum_{s'\in\cS,a'\in\cA}\mu(s',a')P(s|s',a') + \xi(s),\,\,\,\forall s\in\cS,\nonumber
\end{eqnarray}
where the upper bound $\mu\leq\frac{{\bf1}_{{|\cS|\!\times|\!\cA|}}}{1-\gamma}$ is a redundant constraint satisfied by all state action occupancy measures. 

In the experiments, we test our algorithm in two networks illustrated by Figure \ref{fig:teamRL-1-1} and Figure \ref{fig:teamRL-2-1}. In particular, the network in Figure \ref{fig:teamRL-2-1} is generated by a stochastic block model, with 4 clusters of size $|\cS_1| = \cdots = |\cS_4| = 60$. For any two nodes from the same cluster, the probability of having a link between them is $p = 0.25$; for any two nodes from   different clusters, the probability of having a link between them is $q = 0.005$. Then a random adjacency matrix is generated accordingly. For both cases, we set the action space to have size $|\cA| = 3$. Once the action space $\cA$, the network structure and nodes clusters are determined, for each $(s,a)\in\cS\times\cA$, the transition probability $P(\cdot|s,a)$ is generated randomly among the neighbourhood of $s$ in the network, and the reward $r(s,a)$ is also randomly created. For all the experiment, we randomly generate the initial state distribution $\xi$ and we take the discount factor to be $\gamma=0.9$. With these generated $P,\gamma$ and $\xi$, we can rewrite the constraint in the form of $\sum_{i=1}^nA_i\mu_i = \xi$, with $\mu_i = \mu(\cS_i,:).$\vspace{-0.1cm}

\begin{figure}[H]
	\centering
	\begin{subfigure}{0.3\textwidth}
		\includegraphics[width=\textwidth]{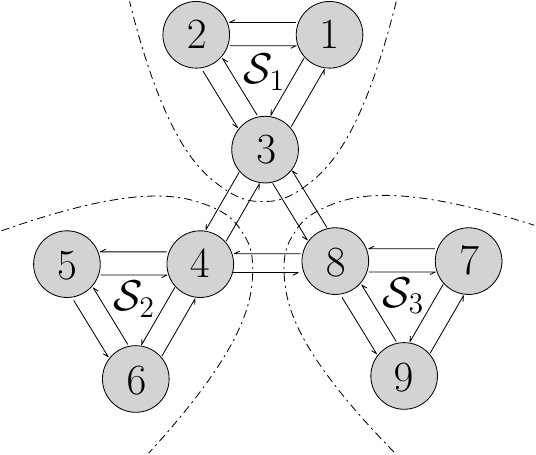}
		\caption{Network structure}
		\label{fig:teamRL-1-1}
	\end{subfigure}
	\hfill\hspace{-0.1cm}
	\begin{subfigure}{0.342\textwidth}
		\includegraphics[width=\textwidth]{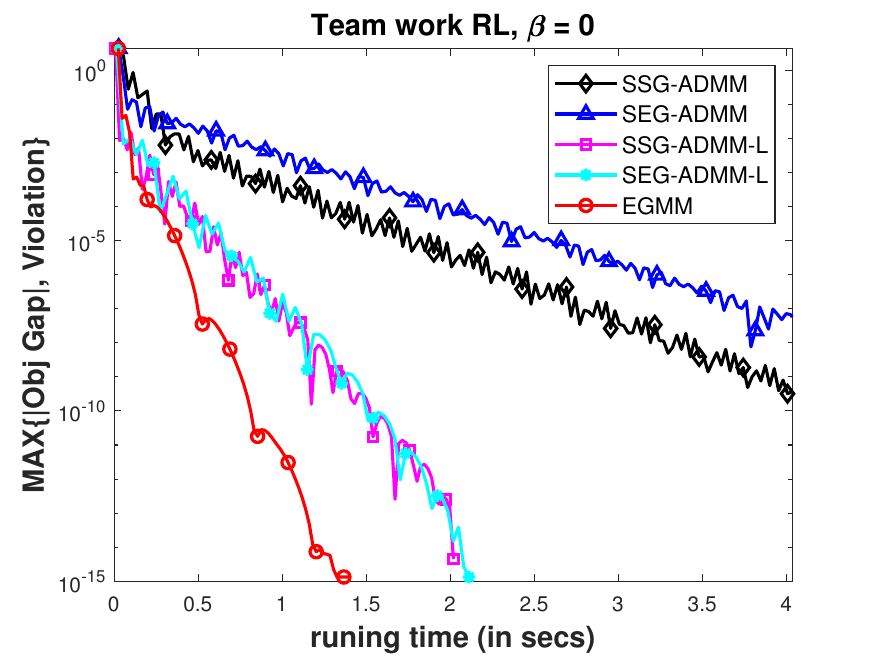}
		\caption{Case $\beta = 0$}
		\label{fig:teamRL-1-2}
	\end{subfigure}
	\hfill\hspace{-0.3cm}
	\begin{subfigure}{0.342\textwidth}
		\includegraphics[width=\textwidth]{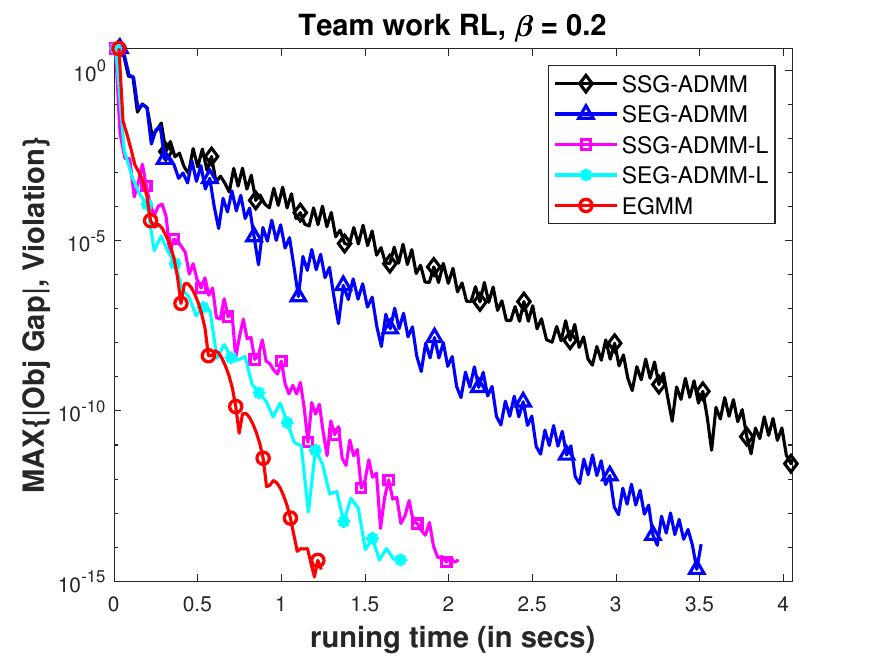}
		\caption{Case $\beta = 0.2$}
		\label{fig:teamRL-1-3}
	\end{subfigure} \vspace{-0.3cm}
	\caption{State space partition structure and experiments with $\beta = 0$ and $\beta = 0.2$.}
	\label{fig:TeamRL-1}
\end{figure} \vspace{-1cm}
\begin{figure}[H] 
	\centering
	\begin{subfigure}{0.28\textwidth}
		\includegraphics[width=\textwidth]{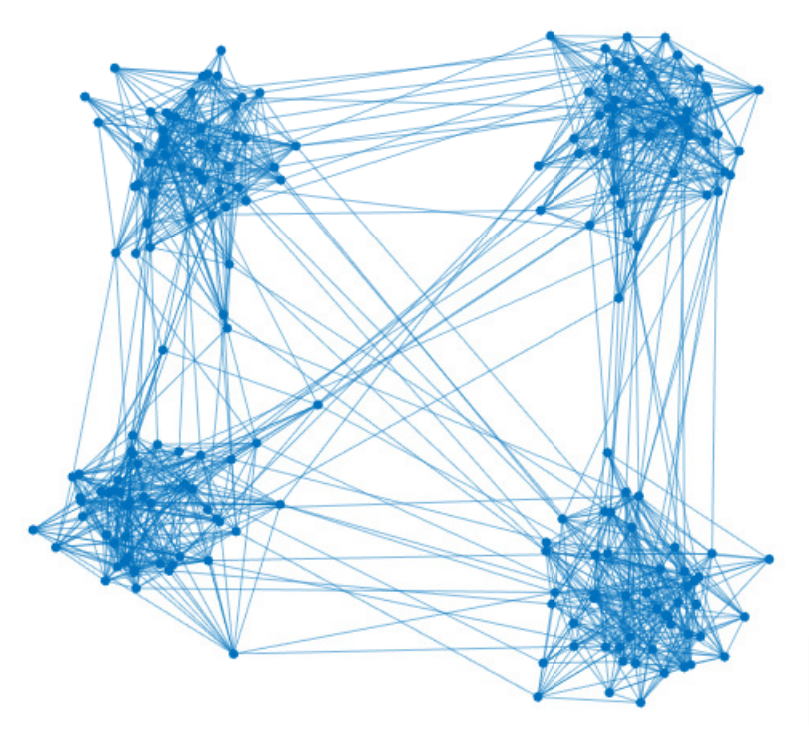}
		\caption{Network structure}
		\label{fig:teamRL-2-1}
	\end{subfigure}
	\hfill\hspace{-0.1cm}
	\begin{subfigure}{0.342\textwidth}
		\includegraphics[width=\textwidth]{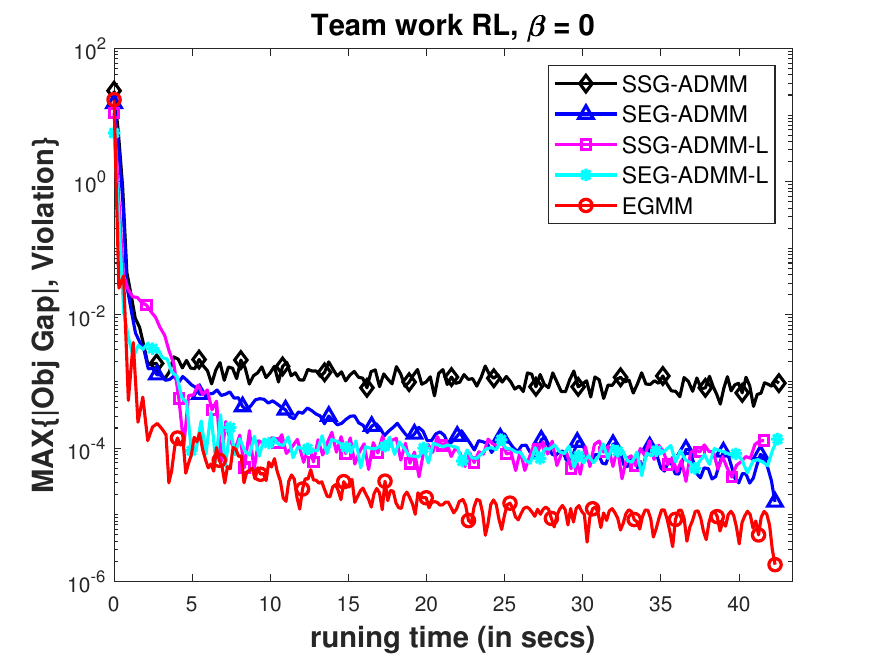}
		\caption{Case $\beta = 0$}
		\label{fig:teamRL-2-2}
	\end{subfigure}
	\hfill\hspace{-0.1cm}
	\begin{subfigure}{0.342\textwidth}
		\includegraphics[width=\textwidth]{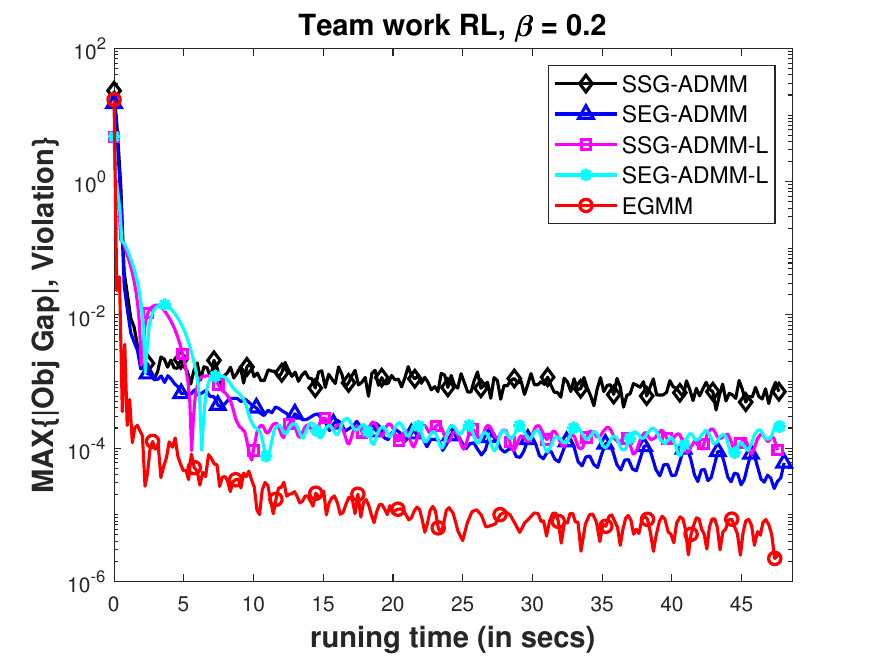}
		\caption{Case $\beta = 0.2$}
		\label{fig:teamRL-2-3}
	\end{subfigure}\vspace{-0.3cm}
	\caption{State space partition structure and experiments with $\beta = 0$ and $\beta = 0.2$.}
	\label{fig:TeamRL-2}
\end{figure}
\vspace{-0.4cm}
For this problem, we test all three proposed algorithms. In particular, for Algorithm \ref{alg:SSG-ADMM} and \ref{alg:SEG-ADMM}, if we simply  set the matrices $H_i = \sigma I$, then we call them SSG-ADMM and SEG-ADMM respectively. In this case, due to the box constraint, the subproblems does not have closed form solution and thus we  solve them with the standard Matlab \texttt{quadprog} function. Let $\gamma_0$ be the penalty coefficient in the augmented Lagrangian. If we set $H_i = \sigma I - \gamma_0 A_i^\top A_i$ to eliminate the quadratic term in the subproblem, we will call them SSG-ADMM-L and SEG-ADMM-L because choosing this specific proximal term is equivalent to linearizing the augmented quadratic penalty term. In this case, the subproblems have closed form solution. And we use EGMM to denote the curve of Algorithm \ref{alg:main}. For all the step size and penalty parameters, they are tuned from $\{1,10,100,1000\}$. For example, for SSG-ADMM, we select the penalty coefficient $\gamma_0$, and the step size $\sigma$ that have best performance from the above set.  For any iteration $(\bar\mu,\bar y)$, the reported error measure is chosen as \vspace{-0.2cm}
$$\max\Bigg\{\bigg|\mathop{\mathrm{maximize}}_{{\bf0}_{{|\cS|\!\times\!|\cA|}}\leq\mu\leq\frac{{\bf1}_{{|\cS|\!\times|\!\cA|}}}{1-\gamma}}\Phi(\mu,\bar y) - \mathop{\mathrm{minimize}}_{\by\geq{\bf0}_n,{\bf1}_n^\top\by=1}\Phi(\bar\mu,y)\bigg|, \bigg\|\sum_{i=1}^nA_i\bar\mu(\cS_i,:)-\xi\bigg\|\Bigg\}.\vspace{-0.2cm} $$
We report the results in Figure \ref{fig:TeamRL-1} and \ref{fig:TeamRL-2}.\vspace{0.6cm}\\
\noindent\textbf{Acknowledgment.} J. Zhang is supported by the Ministry of Education (MOE), Singapore, project WBS No. R-266-000-158-133. M. Wang is supported by NSF grants DMS-1953686, IIS-2107304, CMMI-1653435, and ONR grant 1006977. M. Hong is supported in part by NSF grants  CIF-1910385 and CMMI-1727757.

\bibliographystyle{siamplain}
\bibliography{admm}

\appendix

\section{Proof of Lemma \ref{assumption:bounded_rho} }
\label{appdx:Slater}
\begin{proof}
	First, consider the optimization problem that defines $p_\by(0)$: 
	$$p_\by(0) \,\,:=\,\, \min_{\bx\in\cX} \,\,\Phi(\bx,\by)\quad\mathrm{s.t.}\quad A\bx - a = 0.$$ Due to the compactness of the non-empty feasible region, as well as the lower semi-continuity of $\Phi(\cdot,\by)$, there exists a minimizer $\bx^*(\by)$ for this problem. Due to the convexity of $\Phi(\cdot,\by)$ and Slater's condition, there is an optimal Lagrangian multiplier $\omega^*_\by\in\mathrm{Span}(A)$ associated with the linear constraint $A\bx = a$ and the strong duality holds. Then the classical Lagrangian multiplier theory and sensitivity analysis for convex optimization tells us that\vspace{-0.0cm} 
	\begin{eqnarray}
		\label{eqn:slater-1}
		\qquad\bx^*(\by) = \argmin_{\bx\in\cX}  \Phi(\bx,\by) - (A\bx-a)^\top\omega_\by^*,\quad A\bx^*(\by)=a,\quad\mbox{and}\quad \omega_\by^*\in\partial p_{\by}(0)\vspace{-0.0cm} 
	\end{eqnarray} 
	Since we assume  $\Phi(\cdot,\cdot)$ is bounded over $\cX\times\cY$,  there exist $\underline{\Phi}$ and $\overline{\Phi}$ such that $$\underline{\Phi} \leq \Phi(\bx,\by) \leq \overline{\Phi},\qquad\mbox{for}\qquad \forall \bx\in\cX, \forall\by\in\cY.$$
	Because $\omega_\by^*\in\mathrm{Span}(A)$, then exists $c$ s.t. $Ac = \omega_\by^*$. In particular, if we pick a $c$ as the minimum norm solution to this linear equation, then we know $\|c\|\leq\frac{\|\omega_\by^*\|}{\sigma_{\min}(A)}$, where $\sigma_{\min}(A)$ is the minimum non-zero singular value of $A$.  Overall, there exists $c$ s.t. \vspace{-0.0cm} 
	\begin{eqnarray}
		\label{eqn:slater-2}
		Ac = \omega_\by^*\qquad\mbox{with}\qquad \|c\|\leq \|\omega_\by^*\|/\sigma_{\min}(A).\vspace{-0.0cm} 
	\end{eqnarray}   
	By Assumption \ref{assumption:slater} (Slater's condition), there $\exists\hat \bx\in\mathrm{int}(\cX)$ s.t. $A\hat{\bx} = a$. Consequently, there $\exists\delta_{\hat \bx}>0$ such that $B(\hat{\bx},\delta_{\hat \bx})\subseteq \cX$.  Combined with \eqref{eqn:slater-2}, we have \vspace{-0.05cm} 
	$$\tilde{\bx}:=\hat{\bx} + \frac{\delta_{\hat\bx}\cdot\sigma_{\min}(A)}{\|\omega_\by^*\|}\cdot c\in B(\hat{\bx},\delta_{\hat\bx})\subseteq \cX.\vspace{-0.05cm} $$ Finally, by \eqref{eqn:slater-1}, we have 
	\begin{eqnarray*}
		\Phi(\bx^*(\by),\by) & = & \min_{\bx\in\cX}  \Phi(\bx,\by) - (A\bx-a)^\top\omega_\by^*\\
		& \leq & \Phi(\tilde\bx,\by) - \left(A\left(\hat{\bx}+\frac{\delta_{\hat\bx}\cdot\sigma_{\min}(A)}{\|\omega_\by^*\|}\cdot c\right)-a\right)^\top\omega_\by^*\\
		& = & \Phi(\tilde\bx,\by) - \delta_{\hat\bx}\cdot\sigma_{\min}(A)\cdot\|\omega_\by^*\|.
	\end{eqnarray*} 
	That is $\|\omega_\by^*\|\leq \frac{\overline{\Phi}-\underline{\Phi}}{\delta_{\hat\bx}\cdot\sigma_{\min}(A)}$ for any $\by\in\cY$. This also indicates that 
	$$\sup_{\by\in\cY} \inf_{\omega_1\in\partial p_\by(0)} \|\omega_1\| \leq \sup_{\by\in\cY} \|\omega_\by^*\|\leq \frac{\overline{\Phi}-\underline{\Phi}}{\delta_{\hat\bx}\cdot\sigma_{\min}(A)}.$$
	Through a completely symmetric analysis, there exists an upper bound 
	$$\sup_{\bx\in\cX} \inf_{\omega_2\in\partial q_\bx(0)} \|\omega_2\| \leq  \frac{\overline{\Phi}-\underline{\Phi}}{\delta_{\hat\by}\cdot\sigma_{\min}(B)},$$ where $\delta_{\hat\by}>0$ is a constant s.t. $B(\hat{\by},\delta_{\hat\by})\subset\cY$, which proves the existence of a finite positive constant $\rho^*$.
\end{proof}

\section{Convergence of multi-block ($N\geq3$)  SEG-ADMM}
\label{appdx:Ngeq3}
Similar to convex optimization, the ADMM-based methods {SSG-ADMM}  and  {SEG-ADMM} in general diverge when $N\geq3$.
In this Appendix, we will consider a partial strong convexity condition \cite{chen2013convergence,lin2015sublinear,lin2015global,li2015convergent,cai2017convergence} for problem \eqref{prob:1-sided}, under which an $\cO(1/T)$ convergence can be derived for {SEG-ADMM}. A perturbation strategy from \cite{lin2016iteration} can be adopted in case this condition does not hold. 

\begin{assumption}
	\label{assumption:strongly-convex}
	$h_i(\cdot,\by)$ is $\mu_i$-strongly convex in $\cX_i$, for $i = 2,...,N$, $\forall \by\in\cY$. 
\end{assumption}
Specifically,  $h_1(\cdot,\by)$ is only required to be convex, instead of strongly convex. Note that Lemma \ref{lemma:SEG-step} is still valid, and we only need to extend Lemma \ref{lemma:2-ADMM-step} as follows. The analysis mostly comes from the proof of \cite{lin2015sublinear}, while we apply the linearization technique from \cite{gao2019randomized} to handle the smooth coupling term $\Psi(\bx,\by)$ in addition. The analysis is very similar to that of Lemma \ref{lemma:2-ADMM-step}, thus we omit the proof of this lemma. 
\begin{lemma}
	\label{lemma:N-ADMM-step}
	Suppose Assumptions  \ref{assumption:Convexity-1sided} and \ref{assumption:strongly-convex} hold. For some fixed $\tilde{\by}\in\cY$ and $H_i\succ0, i=1,...,N$, suppose 
	$(\bx^{k+1},\lambda^{k+1}) = \emph{\texttt{Prox-ADMM}}(\bx^k,\lambda^k;\tilde \by;\gamma,\{H_i\}_{i=1}^N)$. Denote $H = \mathrm{Diag}(H_1,\cdots,H_N)$. For $\forall\mathbf{x}\in\cX$ s.t. $A\bx = a$, we have 
	\begin{eqnarray*} 
		& & \Phi(\bx^{k+1}, \tby) - \Phi(\bx,\tby) - \langle\lambda, A\bx^{k+1}-a\rangle \nonumber\\
		& \leq &  \frac{1}{2\gamma}\left(\|\lambda-\lambda^k\|^2-\|\lambda-\lambda^{k+1}\|^2\right) + \frac{1}{2}\left(\|\bx - \bx^k\|^2_H - \|\bx - \bx^{k+1}\|^2_H\right) \\
		& &  +\frac{\gamma}{2}\!\sum_{j=2}^{N}\!\left(\!\left\|\cA\!\left(\bx_{1:j-1},\bx_{j:N}^k\right)\!-\!a\right\|^2 \!-\! \big\|\cA\big(\bx_{1:j-1},\bx_{j:N}^{k+1}\big)\!-\!a\big\|^2\right) \!-\! \frac{1}{2}\|\bx^{k}\!-\!\bx^{k+1}\|^2_{H-L_\bx I}\nonumber\\
		&& -\frac{1}{2}\sum_{i=2}^N(\mu_i - \gamma N(N-1)\|A_i\|^2)\|x_i-x_i^{k+1}\|^2.
	\end{eqnarray*}	 
\end{lemma}
By Lemma \ref{lemma:SEG-step} and setting $\tby  = \hat \by^{k+1}$ in Lemma \ref{lemma:N-ADMM-step}, we have the following theorem whose proof is omitted.
\begin{theorem}
	\label{theorem:ADMM-SC}
	Suppose Assumptions \ref{assumption:Convexity-1sided}, \ref{assumption:strongly-convex} and \ref{assumption:Lipschitz-1sided} hold, and $N\geq3$. Let
	$(\bar{\bx} ,\bar \by)$ be the output of  Algorithm \ref{alg:SEG-ADMM} after $T$ iterations. As long as we choose $H_i \succeq (L_\bx + L_\by)I$, $i=1,\cdots,N$, $G \succeq L_\by I$, and $\gamma =  \frac{1}{N(N-1)}\min\Big\{\frac{\mu_2}{\|A_2\|^2},...,\frac{\mu_N}{\|A_N\|^2}\Big\}$, it holds for $\forall\rho>0$ that
	\begin{align*}
		&\quad\max_{\by\in \cY} \Phi(\bar {\bx},\by) - \min_{\bx\in\cX,A\bx = a}\Phi(\bx,\bar \by)  + \rho\|A\bar{\bx} - a\|\\
		& \!\leq\! \cO\left(\frac{N^2\rho^2/\gamma + \|G\|\cdot D_\cY^2 + \|H\|\cdot D_\cX^2 + \gamma\sum_{j=2}^N(2N+j)(j-1)\|A_j\|^2D_{\cX_j}^2}{T}\right)
	\end{align*}
	By Lemma \ref{lemma:opt-cond}, it takes $T=\cO(\epsilon^{-1})$ iterations to reach an $\epsilon$-saddle point. 
\end{theorem}
When Assumption \ref{assumption:strongly-convex} does not hold, we can apply the strongly convex $\epsilon$-perturbation strategy of \cite{lin2016iteration}, where problem \eqref{prob:1-sided} is modified as 
\begin{equation}
	\label{prob:perturb}
	\min_{\bx\in\cX} \max_{\by\in\cY} \Phi_\epsilon(\bx, \by) := \Phi(\bx,\by)+ \frac{\epsilon}{2}\sum_{i=2}^N\|x_i-x_i^0\|^2\quad\mbox{s.t.}\quad A\bx = a.
\end{equation}
Therefore, $\Phi_\epsilon$ satisfies Assumption \ref{assumption:strongly-convex}. By properly choosing the parameters, we have the following corollary. 
\begin{corollary}
	\label{corollary:SEG-ADMM-pert}
	Suppose Assumptions \ref{assumption:Convexity-1sided} and \ref{assumption:Lipschitz-1sided} hold, and $N\geq3$. Suppose  $(\bar{\bx}, \bar{\by})$ is generated by applying  Algorithm \ref{alg:SEG-ADMM} to the perturbed problem \eqref{prob:perturb} for $T$ iterations, with the parameters chosen as $\epsilon = \cO(1/\sqrt{T})$, $\gamma =  \frac{\epsilon}{N(N-1)\cdot \max_{2\leq i\leq N}\{\|A_i\|^2\}}$ and $H_i \succeq (L_\bx+L_\by)I, i=1,\cdots,N$, and $G\succeq L_\by I$.  
	Then for any $\rho>0$, we have 
	\begin{eqnarray*} 
		\max_{\by\in \cY} \Phi(\bar {\bx},\by)  -  \min_{\overset{\bx\in\cX}{A\bx = a}} \Phi(\bx,\bar \by)  + \rho\|A\bar{\bx} - a\| \leq \cO\left(\frac{1}{\sqrt{T}}\right).
	\end{eqnarray*}
	Therefore, it takes $T = \cO(\epsilon^{-2})$ iterations to reach an $\epsilon$-saddle point. 
\end{corollary}
\begin{proof}
	As a direct corollary of Theorem \ref{theorem:ADMM-SC}, we have
	$$\max_{\by\in\cY} \Phi_\epsilon(\bar {\bx},\by) - \min_{\overset{\bx\in\cX}{A\bx = a}} \Phi_\epsilon(\bx,\bar \by)  +  \rho\|A\bar{\bx} -a\| \leq \cO\left(\frac{1}{\sqrt{T}}\right).$$
	Then by the compactness of $\cX_i$'s and $\epsilon = \cO(1/\sqrt{T})$, we know that 
	$$\left|\max_{\by\in \cY} \Phi_\epsilon(\bar {\bx},\by)-\max_{\by\in \cY} \Phi(\bar {\bx},\by)\right|\leq \cO\left(\frac{\sum_{i=2}^ND_i^2}{2\sqrt{T}}\right) $$
	and 
	$$\left|\min_{\overset{\bx\in\cX}{A\bx = a}}\Phi_\epsilon(\bx,\bar \by)-\min_{\overset{\bx\in\cX}{A\bx = a}}\Phi(\bx,\bar \by)\right|\leq \cO\left(\frac{\sum_{i=2}^ND_i^2}{2\sqrt{T}}\right),$$
	combining the above inequalities proves the corollary. 
\end{proof}

Note that such perturbation significantly deteriorates the convergence rate of the $\bx$ variable. Therefore, it is not necessary to apply the extra-gradient step to the $\by$-update to accelerate the convergence of $\by$ variable. We can also directly apply the  SSG-ADMM method to the perturbed problem, which still yields the $\cO(1/\sqrt{T})$ convergence rate without requiring the differentiability of $\Phi(\bx,\cdot)$. We summarize the result in the following corollary without a proof. 
\begin{corollary}
	\label{corollary:SSG-ADMM-pert}
	Suppose Assumptions \ref{assumption:Convexity-1sided} and \ref{assumption:bounded-supgrad-1sided} hold, and $N\geq3$. Suppose  $(\bar{\bx} ,\bar{\by})$ is generated by running   Algorithm \ref{alg:SSG-ADMM} to the perturbed problem \eqref{prob:perturb} after $T$ iterations, with the parameters chosen as $\epsilon = \cO(1/\sqrt{T})$, $\gamma =  \frac{\epsilon}{N(N-1)\cdot \max_{2\leq i\leq N}\{\|A_i\|^2\}}$ and $H_i \succeq L_\bx I, i = 1,\cdots,N$, and $G\succeq L_\by I$.  
	Then for any $\rho>0$, we have 
	\begin{eqnarray*} 
		\max_{\by\in \cY} \Phi(\bar {\bx},\by)  -  \min_{\overset{\bx\in\cX}{A\bx = a}} \Phi(\bx,\bar \by)  + \rho\|A\bar{\bx} - a\| \leq \cO\left(\frac{1}{\sqrt{T}}\right).
	\end{eqnarray*}
	Therefore, it takes $T=\cO(\epsilon^{-2})$ iterations to reach an $\epsilon$-saddle point. 
\end{corollary}

\section{Proof of Lemma \ref{lemma:2-ADMM-step}}
\label{appdx:lemma:2-admm-step}
\begin{proof}
	First, the KKT condition of  block $x_1^{k+1}$:
	\begin{eqnarray}
		x_1^{k+1} &=& \argmin_{x_1\in\cX_1} \,\, h_1(x_1,\tilde\by) + \langle\nabla_{x_1}\Psi(\bx^k,\tilde{\by}), x_1 - x_1^k\rangle - \langle\lambda^k,A_1x_1+A_2x_2^k-a\rangle \nonumber\\
		&&\qquad\quad\,\,+ \frac{\gamma}{2}\|A_1x_1+A_2x_2^k-a\|^2 + \frac{1}{2}\|x_1-x_1^k\|^2_{H_1}\nonumber
	\end{eqnarray}
	gives
	\begin{equation*}
		\langle u_1 + \nabla_{x_1}\Psi(\bx^k,\tilde{\by}) -  A_1^\top\lambda^k + \gamma A_1^\top(A_1x_1^{k+1} + A_2x_2^k-a) + H_1(x_1^{k+1}-x_1^k),x_1-x_1^{k+1}\rangle\geq0, 
	\end{equation*}
	for $\forall x_1\in\cX_1$, where $u_1\in\partial_{x_1} h_1(x_1^{k+1},\tilde{\by})$.  Applying the update rule of $\lambda^{k+1}$ and rearranging the terms indicates for $\forall x_1\in\cX_1$ that 
	\begin{align}
		\label{lm:ADMM-step-1}
		\quad\quad\langle u_1 + \nabla_{x_1}\Psi(\bx^k,\tilde{\by}),x_1^{k+1}-x_1\rangle \leq &\,\, \langle \gamma A_1^\top A_2(x_2^k-x_2^{k+1}) - A_1^\top\lambda^{k+1},x_1-x_1^{k+1}\rangle\nonumber\\
		&\,\, +  \langle H_1(x_1^{k+1}-x_1^k),x_1-x_1^{k+1}\rangle.
	\end{align} 
	Similarly, the KKT condition of the block $x_2^{k+1}$ gives 
	\begin{eqnarray}
		\label{lm:ADMM-step-2}
		&&\langle u_2 + \nabla_{x_2}\Psi(\bx^k,\tilde{\by}), x_2^{k+1}-x_2\rangle \\
		&\leq&  - \langle A_2^\top\lambda^{k+1},x_2\!-\!x_2^{k+1}\rangle \!+\! \langle H_2(x_2^{k+1}\!-\!x_2^k), x_2 \!-\! x_2^{k+1}\rangle, \forall x_2\in\cX_2\nonumber,
	\end{eqnarray}
	where $u_2\in\partial_{x_2} h_2(x_2^{k+1},\tilde{\by})$.
	Note that we have assumed that $x_1$ in \eqref{lm:ADMM-step-1} and $x_2$ in \eqref{lm:ADMM-step-2} satisfy $A_1x_1+A_2x_2-a = 0$. Summing up the above two inequalities and applying the convexity of $\Phi(\cdot,\tilde{\by})$ yields  
	
	\begin{eqnarray}
		\label{lm:ADMM-step-3}
		& & \Phi\left(\bx^{k+1},\tilde{\by}\right) - \Phi\left(\bx,\tilde \by\right)\nonumber\\
		& =  & h_1\left(x_1^{k+1}\!,\tilde{\by}\right) \!+\! h_2\left(x_2^{k+1}\!,\tilde{\by}\right) \!+\! \Psi(\bx^{k+1}\!,\tilde{\by})\!-\! h_1\left(x_1,\tilde{\by}\right) \!-\! h_2\left(x_2,\tilde{\by}\right) \!-\! \Psi(\bx,\tilde{\by})\nonumber\\
		& \overset{(i)}{\leq} & \left\langle u_1,x_1^{k+1}\!-\!x_1\right\rangle \!+\! \left\langle u_2, x_2^{k+1}\!-\!x_2\right\rangle \!+\! \Psi(\bx^{k+1}\!,\tilde{\by}) \!-\! \Psi(\bx^{k}\!,\tilde{\by}) \!+\! \Psi(\bx^{k},\tilde{\by}) \!-\! \Psi(\bx,\tilde{\by})\nonumber\\	
		& \overset{(ii)}{\leq} & \left\langle u_1 \!+\! \nabla_{\!x_1}\!\Psi(\bx^k\!,\tilde{\by}),x_1^{k+1}\!-\!x_1\right\rangle \!+\! \left\langle u_2 \!+\! \nabla_{x_2}\Psi(\bx^k\!,\tilde{\by}), x_2^{k+1}\!-\!x_2\right\rangle \!+\!\frac{L_\bx}{2}\|\bx^{k+1}\!-\!\bx^k\|^2\nonumber\\		 
		& \leq & \left\langle \gamma A_1^\top\! A_2\left(x_2^k\!-\!x_2^{k+1}\right) \!-\! A_1^\top\lambda^{k+1},x_1\!-\!x_1^{k+1}\right\rangle \!-\! \left\langle A_2^\top\lambda^{k+1},x_2\!-\!x_2^{k+1}\right\rangle \nonumber\\
		& & \!+   \!\left(x_1^{k+1}\!-\!x_1^k\right)^{\!\!\top}\!\!H_1\!\left(x_1\!-\!x_1^{k+1}\right) \!+\! \!\left(x_2^{k+1}\!-\!x_2^k\right)^{\!\!\top}\!\!H_2\!\left(x_2 \!-\! x_2^{k+1}\right)\!+\!\frac{L_\bx}{2}\|\bx^{k+1}\!-\!\bx^k\|^2\nonumber\\
		& \overset{(iii)}{=} & \lambda^{\!\top}\!\!\left(A_1x_1^{k+1}\!+\!A_2x_2^{k+1}\!-\!b\right)\! + \! \frac{1}{\gamma}\left(\lambda\!-\!\lambda^{k+1}\right)^{\!\!\top}\!\!\left(\lambda^{k+1}\!-\!\lambda^k\right)\! + \!\left(x_1^{k+1}\!-\!x_1^k\right)^{\!\!\top}\!\!H_1\!\left(x_1\!-\!x_1^{k+1}\right)  \nonumber\\
		& & \qquad+\gamma\!\left(x_1\!-\!x_1^{k+1}\right)^{\!\!\top}\!\! A_1^\top\! A_2\!\left(x_2^k\!-\!x_2^{k+1}\right) + \left(x_2^{k+1}\!-\!x_2^k\right)^{\!\top}\!\!H_2\!\left(x_2 \!-\! x_2^{k+1}\right),
	\end{eqnarray}
	where (i) is due to the convexity of $h_1(\cdot,\tilde{\by})$ and $h_2(\cdot,\tilde{\by})$, (ii) is because 
	$$\Psi(\bx^{k+1},\tilde{\by}) - \Psi(\bx^{k},\tilde{\by})\leq \langle \nabla_{\bx}\Psi(\bx^k,\tilde{\by}), \bx^{k+1}-\bx^k \rangle + \frac{L_\bx}{2}\|\bx^{k+1}-\bx^k\|^2$$ 
	$$\Psi(\bx^{k},\tilde{\by}) - \Psi(\bx,\tilde{\by})\leq\langle \nabla_{\bx}\Psi(\bx^k,\tilde{\by}),\bx^k-\bx\rangle,$$
	and (iii) is because $A_1x_1+A_2x_2=a$.   
	Note that for any vectors $a,b,c,d$, and matrix $H\succ0$, it holds that 
	\begin{eqnarray}
		\left(a-b\right)^\top\!\! H(c-d) & = & \left(\sqrt{H}a-\sqrt{H}b\right)^\top \left(\sqrt{H}c-\sqrt{H}d\right)\nonumber\\
		& = & \frac{1}{2}\left(\|a-d\|^2_H - \|a-c\|^2_H\right) + \frac{1}{2}\left(\|c-b\|^2_H - \|d-b\|^2_H\right)
	\end{eqnarray} 
	We can bound the terms in \eqref{lm:ADMM-step-3} by 
	\begin{align*}
		\left(\lambda\!-\!\lambda^{k+1}\right)^{\!\top}\!\!&\left(\lambda^{k+1} - \lambda^k\right) = \frac{1}{2}\left(\|\lambda\!-\!\lambda^k\|^2 \!-\! \|\lambda\!-\!\lambda^{k+1}\|^2 \!-\!\|\lambda^k\!-\!\lambda^{k+1}\|^2\right)\\
		&=  \frac{1}{2}\left(\|\lambda-\lambda^k\|^2 - \|\lambda-\lambda^{k+1}\|^2\right) - \frac{\gamma^2}{2}\|A_1x_1^{k+1}+A_2x_2^{k+1}-a\|^2,
	\end{align*} 
	\begin{eqnarray*}
		\left\langle H_1\left(x_1^{k+1}-x_1^k\right),x_1-x_1^{k+1}\right\rangle = \frac{1}{2}\left(\|x_1-x_1^{k}\|^2_{H_1} - \|x_1-x_1^{k+1}\|^2_{H_1} - \|x^k_1-x_1^{k+1}\|^2_{H_1}\right),
	\end{eqnarray*} 
	\begin{eqnarray*}
		\left\langle H_2(x_2^{k+1}-x_2^k),x_2-x_2^{k+1}\right\rangle & = & \frac{1}{2}\left(\|x_2-x_2^{k}\|^2_{H_2} - \|x_2-x_2^{k+1}\|^2_{H_2} - \|x_2^k-x_2^{k+1}\|^2_{H_2}\right),
	\end{eqnarray*}
	and
	\begin{eqnarray*}
		(x_1-x_1^{k+1})^\top A_1^\top A_2(x_2^k-x_2^{k+1})
		& = & \left[(A_1x_1-a) - (A_1x_1^{k+1}-a)\right]^\top\left[ - A_2x_2^{k+1} - (-A_2x_2^k)\right]\nonumber  \\
		& = & \frac{1}{2}\|A_1x_1+A_2x_2^k-a\|^2 - \frac{1}{2}\|A_1x_1+A_2x_2^{k+1}-a\|^2 \\
		& & + \frac{1}{2}\|A_1x_1^{k+1} + A_2x_2^{k+1}-a\|^2 - \frac{1}{2}\|A_1x_1^{k+1}+A_2x_2^k-a\|^2.\nonumber
	\end{eqnarray*}
	Substituting the above bounds into \eqref{lm:ADMM-step-3} proves the Lemma.
\end{proof}

\section{Proof of Lemma \ref{lemma:SG-step}}
\label{appdx:lemma:SG-step} 
\begin{proof}
	First, the KKT condition of %$\max_{\by\in \cY}\,\,\langle u^k,\by\!-\!\by^k\rangle \!-\! \frac{\sigma_\by}{2}\|\by\!-\!\by^k\|^2$ 
	$\min_{\by\in \cY}\,\,\frac{1}{2} \left\|y - [y^k + G^{-1}\cdot u^k] \right\|^2_G$ gives
	\begin{eqnarray} 
		\label{lm:SG-step-1}
		\langle u^k, \by-\by^{k+1}\rangle &\leq& \langle G( \by^{k+1}-\by^k),\by- \by^{k+1}\rangle, \forall \by\in\cY.
	\end{eqnarray}
	By the concavity of $\Phi(\bx^{k+1},\cdot)$ and the fact that $u^k\in\partial_\by \Phi(\bx^{k+1},\by^k)$, we have 
	\begin{align}
		\label{lm:SG-step-2}
		\Phi(\bx^{k+1},\by) - \Phi(\bx^{k+1}&,\by^k)  \leq  \langle u^k,\by-\by^k\rangle\\
		= & \langle u^k,\by-\by^{k+1}\rangle + \langle u^k,\by^{k+1}-\by^k\rangle\nonumber\\
		\leq & \langle u^k,\by-\by^{k+1}\rangle + \frac{1}{2}\|u^k\|^2_{G^{-1}}  + \frac{1}{2}\|\by^{k+1}-\by^k\|^2_G\nonumber.
	\end{align} 
	We also have the following identity that 
	\begin{equation}
		\label{lm:SG-step-3}
		\langle G(\by^{k+1}-\by^k),\by-\by^{k+1}\rangle =  \frac{1}{2}\left(\|\by-\by^{k}\|^2_G - \|\by-\by^{k+1}\|^2_G - \|\by^k-\by^{k+1}\|^2_G\right).
	\end{equation}
	Combining \eqref{lm:SG-step-1}, \eqref{lm:SG-step-2} and \eqref{lm:SG-step-3}, we have \vspace{-0.05cm} 
	\begin{equation*}
		\Phi(\bx^{k+1},\by) - \Phi(\bx^{k+1},\by^k) \leq \frac{1}{2}\|\by-\by^{k}\|^2_G - \frac{1}{2}\|\by-\by^{k+1}\|^2_G + \frac{1}{2}\|u^k\|^2_{G^{-1}},\vspace{-0.05cm} 
	\end{equation*}
	which proves the lemma.
\end{proof}

\section{Proof of Lemma \ref{lemma:SEG-step}}
\label{appdx:lemma:SEG-step}
\begin{proof}
	The optimality  of 
	$\hat \by^{k+1} \!\!=\!\! \argmax_{\by\in \cY}\langle\nabla_\by \Phi(\bx^k\!,\!\by^k),\by\!-\!\by^k\rangle \!-\! \frac{1}{2}\|\by\!-\!\by^k\|^2_G$
	gives
	\begin{equation*} 
		-\langle\nabla_\by \Phi(\bx^k\!,\!\by^k),\hat \by^{k+1}-\by\rangle \leq \langle G(\hat \by^{k+1}-\by^k),\by-\hat \by^{k+1}\rangle, \quad\forall \by\in\cY.
	\end{equation*}
	By setting $\by = \by^{k+1}$ in the above inequality yields
	\begin{equation}
		\label{lm:SEG-step-1}
		-\langle\nabla_\by \Phi(\bx^k,\by^k),\hat \by^{k+1}-\by^{k+1}\rangle \leq \langle G(\hat \by^{k+1}-\by^k),\by^{k+1}-\hat \by^{k+1}\rangle.
	\end{equation}
	The optimality  of $\by^{k+1} = \argmax_{\by\in \cY}\langle\nabla_\by \Phi(\bx^{k+1},\hat \by^{k+1}),\by-\by^k\rangle - \frac{1}{2}\|\by- \by^k\|^2_G$
	gives
	\begin{eqnarray}
		\label{lm:SEG-step-2}
		-\langle\nabla_\by \Phi(\bx^{k+1},\hat \by^{k+1}),\by^{k+1}-\by\rangle \leq \langle G(\by^{k+1} - \by^k),\by - \by^{k+1}\rangle,\,\,\,\,\forall \by\in\cY.
	\end{eqnarray}
	The concavity of $\Phi(\bx^{k+1},\cdot)$ indicates that 
	\begin{eqnarray}
		& & \Phi(\bx^{k+1},\by)  - \Phi(\bx^{k+1},\hat \by^{k+1})\nonumber\\
		&\leq& -\langle \nabla_\by \Phi(\bx^{k+1},\hat \by^{k+1}),\hat \by^{k+1}-\by\rangle\nonumber\\
		& = & -\langle \nabla_\by \Phi(\bx^{k+1},\hat \by^{k+1}),\by^{k+1} - \by\rangle  -  \langle \nabla_\by \Phi(\bx^{k}, \by^{k}),\hat \by^{k+1} - \by^{k+1}\rangle\nonumber\\
		& &  + \langle \nabla_\by \Phi(\bx^{k}, \by^{k}) - \nabla_\by \Phi(\bx^{k+1},\hat \by^{k+1}),\hat \by^{k+1}-\by^{k+1}\rangle\nonumber\\
		& \leq & -\langle \nabla_\by \Phi(\bx^{k+1},\hat \by^{k+1}),\by^{k+1} - \by\rangle  -  \langle \nabla_\by \Phi(\bx^{k}, \by^{k}),\hat \by^{k+1} - \by^{k+1}\rangle\nonumber\\
		& &  + \frac{1}{2L_\by}\|\nabla_\by \Phi(\bx^{k}, \by^{k}) - \nabla_\by \Phi(\bx^{k+1},\hat \by^{k+1})\|^2 + \frac{L_\by}{2}\|\hat \by^{k+1}-\by^{k+1}\|^2\nonumber\\
		& \overset{(i)}{\leq}  &\langle G(\hat \by^{k+1}-\by^k),\by^{k+1}-\hat \by^{k+1}\rangle+ \langle G(\by^{k+1}-\by^k),\by-\by^{k+1}\rangle\nonumber\\
		&& + \frac{L_\by}{2}\cdot\Big(\|\bx^k - \bx^{k+1}\|^2 + \|\by^k - \hat\by^{k+1}\|^2 + \|\by^{k+1} - \hat{\by}^{k+1}\|^2\Big) \nonumber\\
		& = & \frac{1}{2}\left(\|\by-\by^k\|^2_G - \|\by-\by^{k+1}\|^2_G\right) +  \frac{L_\by}{2}\|\bx^{k+1}-\bx^k\|^2 \nonumber\\
		&&- \frac{1}{2}\Big(\|\hat \by^{k+1}-\by^{k}\|^2_{G-L_\by I} +\|\hat \by^{k+1}-\by^{k+1}\|^2_{G-L_\by I}\Big) ,\nonumber
	\end{eqnarray}
	where (i) is due to Assumption \ref{assumption:Lipschitz-1sided}, and \eqref{lm:SEG-step-1} and \eqref{lm:SEG-step-2}. This completes the proof. 
\end{proof}

\end{document}